                    \def\version{May 15, 2020}                       %

 \documentclass[reqno,11pt]{amsart}
 \usepackage{amsmath, amsthm, a4, latexsym, amssymb}
\usepackage[unicode]{hyperref}
\usepackage{color}

 \usepackage{mathrsfs}

\def\@rmrk#1#2{\refstepcounter
    {#1}\@ifnextchar[{\@yrmrk{#1}{#2}}{\@xrmrk{#1}{#2}}}

\makeatletter\@addtoreset{equation}{section}\makeatother

 \newfont{\bfit}{cmbxti10 scaled 1200}

\renewcommand{\d}{{\rm d}}

\newcommand{\cC}{{\mathfrak C}}

\newcommand{\cF}{{\mathcal F}}

\newcommand{\cM}{{\mathcal M}}

\newcommand{\cW}{{\mathcal W}}

 \newcommand{\e}{{\rm e} }

 \newcommand{\eps}{\varepsilon}

 \newcommand{\R}{\mathbb{R}}
  \newcommand{\Q}{\mathbb{Q}}
 \newcommand{\N}{\mathbb{N}}

 \newcommand{\E}{\mathbb{E}}
 \renewcommand{\P}{\mathbb{P}}
 \newcommand{\weak}{\Rightarrow}
 \newcommand{\vague}{\hookrightarrow} 
 \def\1{{\mathchoice {1\mskip-4mu\mathrm l} 
{1\mskip-4mu\mathrm l}
{1\mskip-4.5mu\mathrm l} {1\mskip-5mu\mathrm l}}}

 \newcommand{\Mcal}{{\mathcal M}}

\newcommand{\heap}[2]{\genfrac{}{}{0pt}{}{#1}{#2}}

\newcommand{\ssup}[1] {{\scriptscriptstyle{({#1}})}}

\renewcommand{\subsection}{\secdef \subsct\sbsect}
\newcommand{\subsct}[2][default]{\refstepcounter{subsection}
\vspace{0.15cm}
{\flushleft\bf \arabic{section}.\arabic{subsection}~\bf #1  }
\nopagebreak\nopagebreak}
\newcommand{\sbsect}[1]{\vspace{0.1cm}\noindent
{\bf #1}\vspace{0.1cm}}

{\nopagebreak {\hfill\rule{2mm}{2mm}}\\ }

\newtheorem{theorem}{Theorem}[section]
\newtheorem{lemma}[theorem]{Lemma}
\newtheorem{cor}[theorem]{Corollary}
\newtheorem{prop}[theorem]{Proposition}

\newtheoremstyle{thm}{1.5ex}{1.5ex}{\itshape\rmfamily}{}
{\bfseries\rmfamily}{}{2ex}{}

\newtheoremstyle{rem}{1.3ex}{1.3ex}{\rmfamily}{}
{\itshape\rmfamily}{}{1.5ex}{}
\theoremstyle{rem}
\newtheorem{remark}{{\slshape\sffamily Remark}}[]

\refstepcounter{subsubsection}

\def\thebibliography#1{\section*{References}
  \list%
  {\arabic{enumi}.}
    {\settowidth\labelwidth{[#1]}\leftmargin\labelwidth
    \advance\leftmargin\labelsep
    \parsep0pt\itemsep0pt
    \usecounter{enumi}}
    \def\newblock{\hskip .11em plus .33em minus .07em}
    \sloppy                   
    \sfcode`\.=1000\relax}

\begin{document}
\title[Identification of the Polaron measure in strong coupling]
{\large Identification of the Polaron measure in strong coupling  and the Pekar variational formula}
\author[Chiranjib Mukherjee and S. R. S. Varadhan]{}
\maketitle
\thispagestyle{empty}
\vspace{-0.5cm}

\centerline{\sc  Chiranjib Mukherjee\footnote{University of M\"unster, Einsteinstrasse 62, M\"unster 48149, Germany, {\tt chiranjib.mukherjee@uni-muenster.de}}
and S. R. S. Varadhan \footnote{Courant Institute of Mathematical Sciences, 251 Mercer Street, New York, NY 10012, USA, {\tt varadhan@cims.nyu.edu}}}
\renewcommand{\thefootnote}{}

\vspace{-0.5cm}
\centerline{\textit{University of M\"unster and Courant Institute New York}}
\vspace{0.2cm}

\begin{center}
\version
\end{center}

\begin{quote}{\small {\bf Abstract: }
The path measure corresponding to the {\it Fr\"ohlich Polaron} appearing in quantum statistical mechanics is defined as the tilted measure $$ \widehat{\mathbb P}_{\varepsilon,T}= \frac{1}{Z(\varepsilon,T)}\exp\bigg(\frac{1}{2}\int_{-T}^T\int_{-T}^T \frac{\varepsilon\mathrm e^{-\varepsilon |t-s|}}{|\omega(t)-\omega(s)|} \mathrm d s \,\mathrm d t\bigg)\mathrm d\mathbb P. $$ Here $\varepsilon>0$ is the Kac parameter (or the inverse-coupling), and $\mathbb P$ is the law of $3d$ Brownian increments. In \cite{MV18} it was shown that the (thermodynamic) limit $\lim_{T\to\infty}\widehat{\mathbb P}_{\varepsilon,T}=\widehat{\mathbb P}_\varepsilon$ exists as a process with stationary increments and this limit was identified explicitly as a mixture of Gaussian processes. In the present article, the strong coupling limit or the vanishing Kac parameter limit $\lim_{\varepsilon\to 0} \widehat{\mathbb P}_\varepsilon$ is investigated. It is shown that this limit exists and coincides with the increments of the Pekar process, which is a stationary diffusion process with generator $\frac 12 \Delta+ (\nabla\psi/\psi)\cdot \nabla$, where $\psi$ is the unique (modulo shifts) maximizer of the Pekar variational problem $$ g_0=\sup_{\|\psi\|_2=1} \Big\{\int_{\mathbb R^3}\int_{\mathbb R^3}\,\psi^2(x) \psi^2(y)|x-y|^{-1}\mathrm d x\mathrm d y -\frac 12\|\nabla \psi\|_2^2\Big\}. $$ As shown in \cite{MV14,KM15,BKM15}, the Pekar process is itself approximated by the limiting ``mean-field Polaron measures", and thus, the present identification of the strong coupling Polaron is a rigorous justification of the '``mean-field approximation" (on the level of path measures) conjectured by Spohn in \cite{S87}. This approximation in the vanishing Kac limit ($\varepsilon\to 0$) is also shown to hold for a general class of Kac-Interaction of the form $H(t,x)=\varepsilon \mathrm e^{-\varepsilon|t|} V(x)$ where $V$ is any continuous function vanishing at infinity.}  
\end{quote}

\section{Introduction  and summary}\label{sec-intro}

Consider any function $V:\R^d\to \R$ which is rotationally symmetric, continuous and vanishes at infinity.  Alternatively, choose the Coulomb potential $V(x)=\frac 1 {|x|}$ in $\R^3$ which explodes at the origin. For any $T>0$, let $\P=\P_T$ be the law of Brownian increments in the time interval $[-T,T]$ (i.e., $\P$ is defined only on the $\sigma$-algebra generated by $\{\omega(t)-\omega(s)\colon -T\leq s < t \leq T\}$) and consider a transformed path measure which descends from a {\it Kac-interaction} of the form 
\begin{equation}\label{PV}
\d \widehat\P_{\eps,T}= \frac{1}{Z(\eps,T)}\,\, \exp\bigg\{\frac{1}{2}\int_{-T}^T\int_{-T}^T
{\eps\e^{-\eps |t-s|}}{V(\omega(t)-\omega(s))} \,\d s \,\d t\bigg\}\d\P
\end{equation}
where $\eps>0$ is a constant, or the {\it Kac-parameter}. Here $Z_{\eps,T}$ is the total mass of the exponential weight above, or the {\it partition function}, whose asymptotic behavior in the thermodynamic limit $T\to\infty$, followed by the vanishing 
Kac-parameter limit $\eps\to 0$, has been investigated in \cite{DV83}, where it is shown that 
\begin{equation}\label{V}
\begin{aligned}
&\lim_{\eps\to 0}\lim_{T\to\infty} \frac 1 T\log Z_{\eps,T} \\
&= \sup_{\heap{\psi\in H^1(\R^d)}{\|\psi\|_2=1} }
\bigg\{\int_{\R^d}\int_{\R^d}\,\psi^2(x) \psi^2(y)V(x-y) \, \d x\d y -\frac 12\|\nabla \psi\|_2^2\bigg\}.
\end{aligned}
\end{equation}
Assuming that the above variational problem admits a maximizer $\psi$ which is unique modulo spatial translations, the main result of the article shows that, the actual path measures $\lim_{\eps\to 0}\lim_{T\to\infty} \widehat\P_{\eps,T}$ 
converge in the vanishing Kac interaction to the increments of a stationary diffusion process with generator $\frac 12 \Delta+ (\nabla\psi/\psi)\cdot \nabla$. 
The incentive of our present work came from proving the above convergence of the path measures for 
 the particular choice $V(x)=\frac 1 {|x|}$ in $d=3$, for which the above variational problem is known to admit a smooth maximizer which is unique, up to translations \cite{L76}.
This choice corresponds to the {\it Fr\"ohlich Polaron} -- a model which enjoys quite some prominence in quantum statistical mechanics.

\subsection{Fr\"ohlich polaron and its path measures.} 
Physical motivation of the Fr\"ohlich polaron originates from studying the effective 
behavior of a slow moving electron coupled to a crystal. For the physical relevance of this model, we refer to
\cite{L33,P49,F72,S87}.  The probabilistic layout of this problem was also founded by Feynman via the {\it{path integral}}
 formulation  which is captured by studying the behavior of a Gibbs measure supported 
on a three dimensional Brownian motion acting under a self-attractive Coulomb interaction -- see also 
\cite[Sect. 1]{MV18} for a discussion which relates the quantum mechanical background of the polaron to the path integral approach 
which we pursued in \cite{MV18} and continue to do so in the present context. 
This self-attractive interaction then defines the tilted measure of the form

\begin{equation}\label{eq-Polaron-eps}
{\widehat\P}_{\eps,T}(\d \omega)= \frac 1 {Z(\eps,T)} \exp\bigg\{ \frac {\eps} {2}\int_{-T}^T\int_{-T}^T \frac{\e^{-\eps\, |t-s|}}{|\omega(t)- \omega(s)|} \d\sigma \d s\bigg\} \,\, \P(\d \omega).
\end{equation}
In the above display, $\eps>0$  is  a constant and if we set 
\begin{equation}\label{eps-alpha}
\eps= \alpha^{-2},
\end{equation}
$\alpha$ is called the {\it  coupling parameter}. As before and also in \eqref{eq-Polaron-eps}, $\P$ refers to the law of three dimensional white noise
which is defined only on the $\sigma$-field generated by three dimensional Brownian increments $\{\omega(t)-\omega(s)\colon -T\leq s < t \leq T\}$, while
$Z(\eps,T)$ is the normalization constant or the {\it{partition function}}. 

Here we are concerned with the physically relevant regime of the {\it{strong coupling limit}} initiated already by Pekar (\cite{P49}).
This regime corresponds to studying the asymptotic behavior of the interaction \eqref{eq-Polaron-eps} as $T\to\infty$, followed by  $\eps\to 0$ (or $\alpha\to\infty$). 
Note that for any $\eps>0$, replacing $\omega(s)$  by
$\sqrt{\eps}\omega(\frac{s}{\eps})$  and invoking the scaling property of Brownian motion, we get
\begin{align*}
Z(\eps, T)&=\E^{\P}\bigg[\exp\bigg\{\frac{\eps}{2} \int_{-T}^T\int_{-T}^T \frac{ \e^{-\eps |t-s|}\,\d s \,\d t}{|\omega(t)-\omega(s)|}\bigg\}\bigg]\\
&=\E^{\P}\bigg[\exp\bigg\{\frac{1}{2\sqrt{\eps}} \int_{-\eps T}^{\eps T}\int_{-\eps T}^{\eps T} \frac{ \e^{-|t-s|}\,\d s \,\d t}{|\omega(t)-\omega(s)|}\bigg\}\bigg] \notag\\
&=Z\big(\frac{1}{\sqrt \eps}, \eps T\big)\\
&{=}Z\bigg(\alpha, \frac {T}{\alpha^2} \bigg)
\end{align*}
where the last identity follows from \eqref{eps-alpha}. It was conjectured in \cite{P49} that the {\it{ground-state energy}} of the strong coupling Polaron
\begin{equation}\label{DV}
\begin{aligned}
g_0\stackrel{\mathrm{(def)}}{=}\lim_{\eps\to 0}\lim_{T\to\infty}\frac{1}{2T}\log Z(\eps,T)&=\lim_{\alpha\to\infty}\lim_{T\to\infty}\frac{1}{2T}\log Z(\alpha, \alpha^{-2}T)\\
&=\lim_{\alpha\to\infty}\frac{1}{\alpha^2}\lim_{T\to\infty}\frac{1}{2T}\log Z(\alpha, T)
\end{aligned}
\end{equation}
exists and is given by the {\it{Pekar variational formula}} 
\begin{equation}\label{Pekarfor}
g_0=\sup_{\heap{\psi\in H^1(\R^3)}{\|\psi\|_2=1}} 
\Bigg\{\int_{\R^3}\int_{\R^3}\d x\d y\,\frac {\psi^2(x) \psi^2(y)}{|x-y|} -\frac 12\big\|\nabla \psi\big\|_2^2\Bigg\}.
\end{equation}
Here $H^1(\R^3)$ denotes the usual Sobolev space of square integrable
functions with square integrable gradient. Pekar's conjecture  was proved in \cite{DV83} using large deviation theory (see also \cite{LT97} for a different approach) and the Pekar variational formula was analyzed by Lieb (\cite{L76}) who showed  that   the supremum  in \eqref{Pekarfor} is attained  
and the maximizing set consists of only translates 
$$
\mathfrak m= \{\psi_y(\cdot)=\psi_0(\cdot-y): y\in R^3\}
$$  of  a single maximizer $\psi_0$, 
which is rotationally symmetric  around $0$.

Apart from the ground state energy, another relevant physical quantity for the Polaron is its so-called {\it{effective mass}}, whose rigorous definition 
requires investigating the (asymptotic) behavior of the Polaron path measures $\widehat\P_{\eps,T}$. Unlike the partition function, a rigorous analysis of the actual path measures 
$\widehat\P_{\eps,T}$ turned out to be much more subtle and had remained unanswered on a rigorous level. In a recent article \cite{MV18} we have shown that 
for $\eps>0$, the limit 
$$
{\widehat \P}_{\eps}=\lim_{T\to\infty}\widehat\P_{\eps,T}
$$
exists and identified the limit $\widehat\P_\eps$ {\it{explicitly}}. 
As a corollary, we have also deduced the central limit theorem for the distributions 
\begin{equation}\label{CLT}
\lim_{T\to\infty}\widehat\P_{\eps,T}\bigg[\frac{\omega(T)-\omega(-T)}{\sqrt{2T}}\in \cdot\bigg]= \lim_{T\to\infty}\widehat\P_{\eps}\bigg[\frac{\omega(T)-\omega(-T)}{\sqrt{2T}}\in \cdot\bigg]
= \mathbf N\big(0,\sigma^2(\eps)\mathbf{Id}\big)
 \end{equation}
of the increment of the process under $\widehat\P_{\eps,T}$ and $\widehat\P_\eps$ as $T\to\infty$ and obtained an explicit formula for the limiting variance $\sigma^2(\eps)\in (0,1)$  which is directly 
related to the aforementioned effective mass $m_{\mathrm{eff}}(\eps)$ of the Polaron.\footnote{The relation $m_{\mathrm{eff}}(\eps)^{-1}=\sigma^2(\eps)$ follows as a direct consequence of our CLT result \eqref{CLT}, see \cite{DS19}. 
It has also been shown recently in \cite{LS19} that $\lim_{\eps\to 0} m_{\mathrm{eff}}(\eps)=\infty$. However, the rate of divergence of the latter quantity is not known.}
It is the goal of the present article to investigate and characterize the strong coupling limit 
$$
{\widehat\Q}=\lim_{\eps\to 0} \widehat \P_\eps =\lim_{\eps\to 0}\,\lim_{T\to\infty} \widehat\P_{\eps,T}
$$
of the Polaron measures. As we will see, this limit $\widehat\Q$ will be determined {\it{uniquely}} by {\it{any}} maximizer $\psi$ of the Pekar variational problem \eqref{Pekarfor}.

\subsection{The mean-field Fr\"ohlich Polaron and the Pekar process.}
Before turning to a more formal description of our main results, it is useful to provide an intuitive interpretation of \eqref{DV} and \eqref{Pekarfor}.
We remark that the interaction appearing in the Polaron problem \eqref{eq-Polaron-eps} is {\it{self-attractive}}.
For fixed $\eps>0$, the measure $\widehat\P_{\eps,T}$ favors 
paths which make $|\omega(t)- \omega(s)|$ small, when $|t- s|$ is not large. In other words, these paths tend to clump together on short time scales. 
However, for strong coupling, this interaction becomes more and more smeared out, and on an intuitive level, in this regime (i.e., $\eps\downarrow 0$) one expects the Polaron interaction 
to resemble the {\it mean-field interaction} given by 
\begin{equation}\label{Phat}
\widehat{\P}_T^{\mathrm{(mf)}}(\d \omega)=\frac 1 { Z_T^{\mathrm{(mf)}}}\, \exp\bigg\{\frac 1{T}\int_0^T\int_0^T \d t \d s \,\frac 1{\big|\omega(t)-\omega(s)\big|}\bigg\} \,\P(\omega).
\end{equation}
The earlier result \eqref{DV} indeed justified
this intuition and underlined the parallel behavior  for the partition functions (on a logarithmic scale) of these two models:
$$
g_0= \lim_{\eps\to 0}\lim_{T\to\infty}\frac 1T\log Z(\eps,T)= \lim_{T\to\infty}\frac 1T\log Z_{T}^{\mathrm{(mf)}}.
$$
Based on the above intuition, Spohn (\cite{S87}) conjectured that the
the strong coupling behavior of the actual Polaron measures $\lim_{\eps\to 0}\lim_{T\to\infty} \, \widehat\Q_{\eps,T}$ should be closely related to the behavior of 
its mean-field counterpart $\lim_{T\to\infty}\widehat{\P}_T^{\mathrm{(mf)}}$. Assuming this conjecture to be true, he also heuristically derived  
(\cite{S87}) the actual decay rate (in leading order) of the diffusion constant 
$\sigma^2(\eps)$ of the central limit theorem appearing in \eqref{CLT} as $\eps\to 0$.

A rigorous analysis of the mean-field model \eqref{Phat} was determined in \cite{BKM15} based on the theory developed in \cite{MV14} and its extension \cite{KM15}.
It was shown in (\cite{BKM15}) that the distribution 
$\widehat\P_T^{\mathrm{(mf)}}\, L_T^{-1}$ of the Brownian occupation measures $L_T=\frac 1 T\int_0^T\delta_{\omega_s} \,\d s $ 
under the mean-field model converges to the distribution 
of a random translation $[\psi_0^2\star \delta_X]\,\d z$ of $\psi_0^2 \, \d z$, with the random shift $X$ having a density $c_0\psi_0 $ 
where $\psi_0$ is the maximizer in \eqref{Pekarfor} centered at $0$ and $c_0$ is the normalizing constant. 
Furthermore, it was also shown (\cite{BKM15}) that the mean-field measures themselves converge 
\begin{equation}\label{mf-convergence}
\widehat\P_T^{\mathrm{(mf)}} \Rightarrow c_0 \, \int_{\R^3} \Q_{\psi_y} \, \psi_0(y) \,\,\d y
\end{equation}
to a mixture with weight $c_0\psi_0(y)\d y$ of the  diffusion  processes  $\Q_{\psi_y}$ with
 generator
\begin{equation}\label{OU}
\frac{1}{2}\Delta+\frac{\nabla\psi_y}{\psi_y}\cdot\nabla
\end{equation}
initialized to start from $0$. The heuristic definition of this diffusion process was set forth in \cite{S87} and was called the {\it{Pekar process}}.

\subsection{Main results: Strong coupling / vanishing Kac limit towards increments of Pekar process.}

Note that the distribution of the {\it{increments}} of the stationary versions of the Pekar process 
$\Q_{\psi_y}$ with generator (\ref{OU}) does not depend on $y$ and defines a unique process $\widehat\Q=\widehat\Q_\psi$ on the space of increments. 
In the  present context, our main result is stated as follows:

\begin{theorem}[Convergence of the Fr\"ohlich Polaron measure in strong coupling]\label{thm1} 
Let $\widehat{\P}_{\eps,T}$ be the path measures for the Fr\"ohlich polaron defined in \eqref{eq-Polaron-eps}, and let $\widehat\Q_\psi$ be the common distribution of the  increments of the stationary Pekar process 
$\Q_{\psi_y}$ with generator \eqref{OU}. Then 
$$
\lim_{\eps\to 0}\, \lim_{T\to\infty} \, \widehat\P_{\eps,T}(\cdot)=\lim_{\eps\to 0}\widehat\P_\eps(\cdot)= \widehat\Q_\psi(\cdot)
$$
\end{theorem} 

Thus, the above result justifies the conjecture posed in \cite{S87} regarding the parallel behavior of 
the measures $\lim_{\eps\to 0}\lim_{T\to\infty}\widehat\P_{\eps,T}$ and $\lim_{T\to\infty}\widehat\P_T^{\mathrm{(mf)}}$.

While we are intrinsically interested in analyzing the case of Fr\"ohlich Polaron, the method of our proof is robust enough to show quite generally that path measures coming from any  Kac-interaction of the  form  \eqref{PV} 
with translation invariance in space converge to the increments of the corresponding mean-field model. More precisely, we have as our second main result:

\begin{theorem}[Convergence of general Kac-interactions towards increments of Pekar-type process]\label{thm2}
Let $V:\R^d \to \R$ be any rotationally symmetric and continuous function vanisihing at infinity and assume that the corresponding variational problem \eqref{V} admits a smooth, strictly positive maximizer $\psi^{\ssup V}\in H^1(\R^d)$ which is unique modulo spatial translations. 
Then we have the following identification of the path measures defined in \eqref{PV} in the vanishing Kac limit: 
$$
\lim_{\eps\to 0}\, \lim_{T\to\infty}\,  \widehat\P_{\eps,T}^{\ssup V}(\cdot)= \widehat\Q_{\psi^{\ssup V}}(\cdot)
$$
where $\widehat\Q_{\psi^{\ssup V}}$ is the common distribution of the  increments of the stationary process 
with generator $\frac 12 \Delta+ \frac{\nabla\psi^{\ssup V}}{\psi^{\ssup V}}\cdot\nabla$.
\end{theorem} 

Let us first briefly outline the method developed here for proving the above results. 

\subsection{Outline of the proof:} In order to provide some guidelines for the reader, we will conclude the introduction with an outline of the proof of our main result. 

Let $\Omega_0$ denote the space of continuous functions on $\omega:\R\to\R^d$ vanishing at the origin.  Then for any $t\in \R$, we have a shift $\theta_t:\Omega_0\to \Omega_0$ defined via $(\theta_t\omega)(\cdot)=\omega(t+\cdot)-\omega(\cdot)$, and 
we can denote by $\Mcal_{\mathrm{si}}(\Omega_0)$  the space of $\theta_t$-invariant probability measures 
on $\Omega_0$, or the space of {\it processes with stationary increments}. Note that we also have an action 
$\theta_t:\Omega_0\otimes \R^d\to \Omega_0\otimes \R^d$ by $\theta_t(\omega,x)=(\omega(t+\cdot)-\omega(t), x+ \omega(t))$. Then we can denote by $\Mcal_{\mathrm{s}}(\Omega_0\otimes \R^d)$ to be the space of $\theta_t$-invariant probability measures on $\Omega_0\otimes\R^d$, or the space of {\it stationary processes}. 

The first main step for our proof is to show that $\widehat\P_\eps\in \mathfrak m_\eps$, where $\mathfrak m_\eps$ is the set of maximizers of the variational problem (for the particular case $d=3$)
\begin{equation}\label{geps}
g(\eps)= \sup_{\Q\in \Mcal_{\mathrm{si}}(\Omega_0)}\bigg[ \E^\Q\bigg(\int_0^\infty \frac{\eps \e^{-\eps t} \, \d t}{|\omega(t)-\omega(0)|}\bigg)- H(\Q|\P)\bigg]
\end{equation}
As mentioned earlier, a  variational formula of the above form for $\lim_{T\to\infty} \frac 1 T \log Z_{\eps,T}$  was first obtained in
\cite{DV83} where the supremum above was taken over all stationary processes in $\Mcal_{\mathrm{s}}(\Omega_0\otimes\R^3)$.
This result
was a consequence of a weak large deviation principle (LDP) for the
empirical process of Brownian motion.
However, in this case, the supremum 
may not be attained. This issue is 
resolved if we exploit the underlying i.i.d. structure of the noise which provides exponential tightness and  a full LDP for the emprical process of Brownian increments. 
In this set up, uniform relative entropy estimates then show that
the variational formula \eqref{geps} is coercive
which gurantees
existence of (at least one) maximizer in $\mathfrak m_\eps$ and moreover $\cup_{\eps<\eps_0}\mathfrak m_\eps$ is also tight. The above strong LDP, combined with the existence of the actual limit $\lim_{T\to\infty}\widehat\P_{\eps,T}=\widehat\P_\eps$ then also shows that $\widehat\P_\eps\in \mathfrak m_\eps$. 

It remains to show that if $\eps_n\to 0$ and $(\Q_n)\subset \mathfrak m_{\eps_n}$ is any sequence of maximimizers such that $\Q_n\weak \Q$ weakly, then $\Q\in \Mcal_{\mathrm{si}}(\Omega_0)$ must be  the distribution of the increments of the stationary Pekar process. The task then splits into two further steps. 
First note that not every process with stationary increments appear as the {\it increments} of another stationary process. For our purposes, we first provide a general criterion that determines when any $\Q\in \Mcal_{\mathrm{si}}(\Omega_0)$ admits this cocycle representation (i.e. any $\Q\in \Mcal_{\mathrm{si}}(\Omega_0)$ appears as the increments of some
$\Q^\prime\in \Mcal_{\mathrm{s}}(\Omega_0\otimes\R^d)$),  see Theorem \ref{lemma:cocycle}. This criterion is formulated in terms of convergence of integrals of continuous  functions vanishing at infinity w.r.t. measures on the function space $\Omega_0$. However, since $\Omega_0$ is not even locally compact, there is no notion of usual vague convergence of measures on this space (determined by convergence of integrals w.r.t. continuous functions vanishing at infinity). We therefore formulate a notion of {\it wea-gue} convergence 
on measures on $\Omega_0\otimes \R^d$, see section \ref{sec:weague} which is conceptually important for the proof of Theorem \ref{lemma:cocycle}.  

To this end, however  we encounter another fundamental problem. The cocycle representation obtained from Theorem \ref{lemma:cocycle} is not unique-- any $\Q\in \Mcal_{\mathrm{si}}(\Omega_0)$ can be written
 as the increments of  an {\it entire orbit} $\widetilde\Q^\prime=\{\Q\star \delta_a\colon a\in \R^d\}$  
for some $\Q^\prime\in \Mcal_{\mathrm{s}}(\Omega_0\otimes\R^d)$, where for any $A\subset \Omega_0$ and $B\subset \R^d$, we define $(\Q^\prime\star\delta_a)[A\otimes B]=\Q^\prime[A\otimes (B-a)]$. 
While identifying any limiting maximizer as the increments of the Pekar process, this non-uniqueness leads to the following 
obstacle. For any sequence $(\Q_n)_n\subset \Mcal_{\mathrm s}(\Omega_0\otimes \R^d)$ even if we assume that 
its marginals $\Q_n^{\ssup 1}$ on $\Omega_0$ form a uniformly tight family, its marginals $\Q_n^{\ssup 2}$ on $\R^d$ might still fail to have a convergent subsequence. Its mass may split and escape into two or more different directions (for instance, $\sum_i p_i\delta_{a_n^{\ssup j}}$ such that $|a_n^{\ssup i}-a_n^{\ssup j}|\to \infty$ for $i\ne j$) or it could totally disintegrate into dust like a Gaussian with large variance, or it could form a mixture of all these {\it widely separated} components. By taking spatial shifts and recovering one such component at a time, in the limit we only imagine an empty, finite or countable collection of orbits of sub-probability meaures, while possibly allowing some mass to totally disintegrate into dust. This intuition leads to a refinement of the method developed in \cite{MV14}. We define the space 
$\mathscr X=\{\Theta=[\xi,\beta]\}$ of all collections $[\xi,\beta]$ where $\xi=\{\widetilde\lambda_j\}$ is am empty, finite or countable collection of orbits 
of sub-probability measures on $\Omega_0\otimes\R^d$ and $\beta$ is a probability measure on $\R^d$ such that $\sum_j \widetilde\lambda^{\ssup 1}(\cdot)\leq \beta(\cdot)$, where $\widetilde\lambda^{\ssup 1}$ is the common marginal of $\lambda$ on $\Omega_0$, see Section \ref{sec:compactify}. We metrize the space $\mathscr X$ which provides a topology that ensures the following crucial property: Let $K\subset\mathscr X$ be such that as $\Theta=[\xi,\beta]$ varies over $K$, $\beta$ varies over a uniformly tight family of probability measures on $\Omega_0$. Then every sequence in $K$ finds a convergent subsequence in that metric in $\mathscr X$. On a technical level, the above recipe then leads to a generalization of the theory developed in \cite{MV14} where the compactification was earlier carried out for measures only in $\R^d$.
In the present context,  this refined compactness result 
for the space $\mathscr X$ holds the key for the identification of any limiting maximizer $\lim_{\eps\to 0}\widehat\P_\eps$ as the distribution of increments of the stationary Pekar process.


\noindent{\it Organization of the  article:} The rest of the article is organized as follows. Section \ref{notation} is entirely devoted to defining topologies on measures 
on suitable spaces, their quotient spaces under group actions and deriving properties of the aforementioned space $\mathscr X$. We will derive several useful properties regarding processes with stationary increments and their cocycle repesentations in Section \ref{sec-stationarity}, while Section \ref{sec-entropy} is devoted to proving relative entropy estimates that provide coercivity properties of the variational formula $g(\eps)$, while
 enabling us to circumvent singularity of the Coulomb potential. Combining all previous arguments, in Section \ref{sec-mainresult} and Section \ref{sec-result} 
we prove our key results showing that $\widehat\P_\eps\in \mathfrak m_\eps$, while as $\eps\to 0$,  any limiting maximizer in $\mathfrak m_\eps$ appears as the distribution of increments of the Pekar process.

\section{Compactification of quotient spaces under group actions}\label{notation}

\subsection{The wea-gue topology.}\label{sec:weague}

In the sequel, for any topological space $Y$, we will write for $\Mcal_{1}(Y)$ and $\Mcal_{\leq 1}(Y)$ to be the spaces of all probability and sub-probability measures on $Y$ respectively. 
We now fix a complete separable metric space $X$. Then both $\Mcal_1(X\otimes \R^d)$ and $\Mcal_{\leq 1}(X\otimes \R^d)$ are equipped with the 
weak topology for which a sequence $\lambda_n$ of (sub)-probability measures converges to $\lambda$, written $\lambda_n\weak\lambda$, 
if and only if 
\begin{equation}\label{convergence}
\int F(x,y) \lambda_n(\d x \, \d y)\to \int F(x,y) \lambda(\d x\d y)
\end{equation}
for any continuous and bounded function $F: X\otimes \R^d \to \R$. The same notion of weak convergence holds also for the space $\Mcal_{\leq 1}(\R^d)$.
However, for the latter case we also have the notion of {\it{vague convergence}} (written $\lambda_n\vague \lambda$) which demands \eqref{convergence} to hold 
for continuous functions $f:\R^d\to \R$ vanishing at infinity or for continuous functions with compact support. Equivalently, the vague convergence can be also obtained by considering the weak topology on (sub-)probability 
measures on the one-point compactification $\overline\R^d=\R^d\cup \{\infty\}$ and removing the mass at $\infty$.  Note that any sequence of sub-probability measures 
on $\R^d$ has a vaguely convergent subsequence, while the weak convergence fails to possess this property.

Since the space $X$ need not be locally compact, there is no notion of vague convergence for measures on
$X\otimes\R^d$. However,  we can again consider the weak topology on $\Mcal_{\leq 1}(X\otimes\overline\R^d)$ and remove the mass at $\infty$, 
which 
leads us to the following notion of {\it{wea-gue}} convergence in this set up: We say that a sequence $\lambda_n$ 
converges to $\lambda$ {\it{weaguely}} in the space $\Mcal_{\leq 1}(X\otimes \R^d)$ if and only if 
$\int_{X\otimes\R^d} F(x,y) \lambda_n(\d x \d y)\to \int_{X\otimes\R^d} F(x,y) \lambda(\d x \d y)$ for all continuous functions $F:X\otimes \R^d\to \R$ such that 
$$
\lim_{|y|\to\infty}\,\sup_{x\in X}\, F(x,y)=0.
$$
We will have several occasions to use the following elementary result, which is an immediate consequence of the aforementioned compactness of vague topology and Prohorov's theorem.
\begin{lemma}\label{lemma-weague}
Any sequence of probability measures on $X\otimes R^d$ with uniformly tight marginals on $X$ will have a subsequence that converges weaguely.\qed
\end{lemma}

\subsection{Quotient space of $\Mcal_{\leq 1}(X\otimes \R^d)$ and the metric space $(\mathscr X, \mathscr D)$.}\label{sec:compactify}

Note that the translation group $\{T_a: a\in \R^d\} $ acts on $X\otimes \R^d$ by mapping 
\begin{equation}\label{eq:translation}
(x,y)\mapsto (x,y+a),
\end{equation}
This action then also induces a map on the space  of measures $\Mcal_{\leq 1}(X\otimes \R^d)$ which we denote by 
$$
\lambda(\d x\, \d y)\mapsto(\lambda\star\delta_a)(\d x \, \d y)\stackrel{\mathrm{(def)}}{=}\lambda\big(\d x \,\, \d(y+a)\big) \qquad x \in X, \,\, a,y\in \R^d.
$$
We can then define an equivalence relation on $\Mcal_{\leq 1}(X\otimes \R^d)$ by setting 
$$
\lambda\sim\lambda^\prime\qquad\mbox{if}\quad \lambda^\prime=\lambda\star\delta_a\qquad\mathrm{for \,\,\, some}\quad a\in \R^d,
$$
which leads to the notion of equivalence classes or {\it{orbits}} which we denote by 
$\widetilde\lambda=\{\lambda\star\delta_a\colon a\in \R^d\}$ and the corresponding quotient space by 
$$
\widetilde\Mcal_{\leq 1}(X\otimes \R^d)= \Mcal_{\leq 1}(X\otimes \R^d)\big/ \sim.
$$
For any $\lambda\in \Mcal_{\leq 1}(X\otimes \R^d)$, $\lambda^{\ssup 1}$
and $\lambda^{\ssup 2}$ are its marginals on $X$ and $\R^d$ respectively. If ${\widetilde\lambda}\in \widetilde\Mcal_{\leq 1}(X\otimes \R^d)$ is  an orbit, 
${\widetilde\lambda}^{\ssup 1}$ is well defined as the common marginal, whereas ${\widetilde\lambda}^{\ssup 2}$ is an orbit of sub-probability measures on $\R^d$.

We enlarge the space $\widetilde\Mcal_{\leq 1}(X\otimes \R^d)$ to
\begin{equation}\label{mathscr-X}
\begin{aligned}
\mathscr X&= \bigg\{\Theta=[\xi,\beta]\colon \xi=\{\widetilde\lambda_j\}_j,\,\,\lambda_j\in \Mcal_{\leq 1}(X\otimes\R^d),\,\beta \in \Mcal_1(X), \\
&\qquad\qquad \qquad\qquad\sum_{{\widetilde\lambda}\in \xi}{\widetilde\lambda}^{\ssup 1}(\cdot)\le \beta(\cdot)\bigg\}.
\end{aligned}
\end{equation}

In other words, $\mathscr X$ consists of all collections of $\Theta=[\xi, \beta]$, where $\beta$ is a probability measure on $X$ and $\xi$ is an empty, finite or countable collection $\{\widetilde\lambda_j\}$ of 
orbits $\widetilde \lambda_j \in \widetilde\Mcal_{\leq 1}(X\otimes \R^d)$ with the property that 
$$
\sum_{{\widetilde\lambda}\in \xi}{\widetilde\lambda}^{\ssup 1}(\cdot)\leq \beta(\cdot).
$$ 
Clearly, we have an embedding 
$$
\widetilde\Mcal_{1}(X\otimes \R^d) \hookrightarrow \mathscr X,
$$
since for any single orbit ${\widetilde\lambda}$ of a probability measure, we have 
 $$
\big[\{\widetilde\lambda\}, {\widetilde\lambda}^{\ssup 1)}\big] \in \mathscr X.
 $$ 

\medskip
We will now define a metric on $\mathscr X$. For any $\Theta_1=[\xi_1,\beta_1], \theta_2=[\xi_2,\beta_2] \in \mathscr X$, we set 
\begin{equation}\label{metric1}
\mathscr D(\Theta_1,\Theta_2)= \mathbf D^\star(\xi_1,\xi_2)+\d(\beta_1,\beta_2)
\end{equation}
where $\mathbf D^\star$ and $\d$ are defined as follows. The definition of $\d$ as a metric on $\Mcal_1(X)$ is straightforward. 
Indeed, we choose a  countable set of continuous functions $\mathfrak F= \{f_j(x)\}$ with $\sup_j\sup_x|f_j(x)|\le 1$ such that 
the existence of the limit $\lim_{n\to\infty}\big|\int f_j(x)\beta_n(\d x)-\int f_j(x)\beta(\d x)|=0$ for every $j$ is equivalent to the weak convergence of $\beta_n\weak\beta$.  On ${\cM_1}(X)$, we then have the metric 
\begin{equation}\label{metric2}
\d(\beta_1, \beta_2)=\sum_j \frac{1}{2^j}\big|\int_X f_j(x)\beta_1(\d x)-\int_X f_j(x)\beta_2(\d x)\big|
\end{equation}
We define $\mathbf D^\star$ as follows. For each $k\ge 2$ we denote by  ${\cC}_k$ the space of functions $W(y_1,\ldots,y_k)$ that satisfy $W(y_1+a,\ldots, y_k+a)=
W(y_1,\ldots, y_k)$ for all $a\in \R^d$  and $\lim_{\sup_{i,j}|y_i-y_j|\to\infty}|W(y_1,\ldots, y_k)|=0$. Since ${\cC}_k$
is separable in the uniform metric, we can choose a countable collection ${\cW}_k\subset {\cC}_k$  of functions $W$
such that they  are uniformly bounded by $1$ and their linear combinations  are dense in ${\cC}_k$. We denote by
${\cW}$ the countable set $\cup_{k\ge 2} {\cW}_k$ and list it as $\{W_j\}$. Then with $\mathfrak F=\{f_j\}_j$ being the basis 
for the metric $\d$ on $\Mcal_1(X)$ above, we can enumerate all the combinations $\{f_j,W_{j^\prime}\}$
as single sequence $\{f_r, W_r\}$, so that each $W_r$ is a function $W(y_1,\ldots, y_{k(r)})$ of $k(r)$ variables. Then for any $\xi=\{\widetilde \lambda_j\}$, we set 
 \begin{equation}\label{metric4}
\Lambda_{f,W}(\xi)= \sum_{{\widetilde\lambda}\in \xi_i}\int f(x_1)\cdots f(x_{k(r)}) W(y_1,\ldots, y_{k(r)})\prod_{i=1}^{k_r}\lambda(\d x_i, \d y_i)
\end{equation} 
and define 
 \begin{equation}\label{metric3}
\mathbf D^\star(\xi_1,\xi_2)=\sum_{r=1}^\infty \frac{1}{2^r}\big|\Lambda_{f_r,W_r}(\xi_1)-\Lambda_{f_r,W_r}(\xi_2)\big|
\end{equation}
Note that the integrals in (\ref{metric4}) depend only on the orbit $\widetilde \lambda$ and therefore $\mathbf D^\star$ is well-defined. 
We now need to justify that $\mathscr D(\theta_1,\theta_2)= \mathbf D^\star(\xi_1,\xi_2) + \d(\beta_1,\beta_2)$ defined in \eqref{metric1} is a metric on $\mathscr X$. 
Since non-negativity and triangle inequality is obvious, we only need to verify that $\mathscr D(\Theta_1,\Theta_2)=0$ implies $\Theta_1=\Theta_2$. 
However, since $\d$ is already a metric in $\Mcal_1(X)$, we only need to verify that $\mathbf D^\star(\xi_1,\xi_2)=0$ forces $\xi_1=\xi_2$. 
The following lemma will guarantee the validity of the last statement. 
\begin{lemma}\label{lemma-metric}
Let $\lambda$, $\gamma$ be  two probability measures on $X\otimes \R^d$ such that for any $k\geq 2$, 
\begin{equation}\label{eq-lemma-metric}
\begin{aligned}
&\int W(y_1,y_2,\ldots,y_k) \prod_{i=1}^k f(x_i) \prod_{i=1}^k \lambda(\d x_i,\d y_i) 
\\
&=
\int W(y_1,y_2,\ldots,y_k) \prod_{i=1}^k f(x_i)\prod_{i=1}^k\gamma(\d x_i,\d y_i)
\end{aligned}
\end{equation}
for all  functions $f\in \mathfrak F$  and  $W\in \mathcal W$.
Then there is $a\in \R^d$ such that $\gamma (A\times B)=\lambda (A\times (B+a)\big)$ for all  measurable $A\subset X$ and $B\subset \R^d$.
\end{lemma}
\begin{proof}
Given \eqref{eq-lemma-metric}, if we let $W\to 1$, the bounded convergence theorem implies that the marginals of $\lambda$ and $\gamma$ on $X$ are the same.  
For any $f\ge 0$, let us define two measures $\lambda_f$ and $\gamma_f$ on $\R^d$ by
$$
\lambda_f(B)={1\over c(f)} \int_{X\otimes B} f(x)\lambda (\d x,\d y), 
\qquad
\gamma_f(B)={1\over c(f)} \int_{X\otimes B} f(x)\gamma (\d x,\d y)
$$
where $c(f)$ is the normalizing constant
$$
c(f)=\int_{X\otimes \R^d} f(x) \lambda(\d x,\d y)=\int_{X\otimes \R^d} f(x) \gamma(\d x,\d y).
$$

It follows that for every $f\in \mathfrak F$, $W\in\mathcal W$ and $k \geq 2$, 
$$
\int  W(y_1,y_2,\ldots,y_k) \lambda_f(\d y_1)\cdots\lambda_f(\d y_k)=
\int W(y_1,y_2,\ldots,y_k) \gamma_f(\d y_1)\cdots\gamma_f(\d y_k)
$$
which implies from \cite[Proof of Theorem 3.1]{MV14} that for each $f\in \mathfrak F$, there is an $a=a(f)\in \R^d$ such that 
\begin{equation}\label{eq1-lemma-metric}
\gamma_f(B)=\lambda_f(A+a).
\end{equation}
We need to show that $a(f)$ is independent of $f$. 
We choose another $g\in \mathfrak F$ and let $\phi_f$, $\phi_g$ and $\phi_{(f+g)/ 2}={1\over 2} [\phi_f+\phi_g]$ denote the characteristic functions of $\lambda_f$, $\lambda_g$ and ${1\over 2}[\lambda_f+\lambda_g]$, respectively.
Likewise, $\psi_f$, $\psi_g$ and $\psi_{(f+g)/ 2}={1\over 2} [\psi_f+\psi_g]$ will denote the characteristic functions of $\gamma_f$, $\gamma_g$ and ${1\over 2}[\gamma_f+\gamma_g]$, respectively. 
Then \eqref{eq1-lemma-metric} implies 
$$
\phi_f (t)=\psi_f(t) \e^{\mathbf i\langle t,a\rangle };\quad  \phi_g (t)=\psi_g(t) \e^{\mathbf i\langle t,b\rangle }; \quad \phi_{f+g\over 2} (t)=\psi_{f+g\over 2}(t) \e^{\mathbf i\langle t,c\rangle }
$$
Thus,
$$
[\psi_f(t)+\psi_f(t)]\, \e^{i\langle t,c\rangle }=  \phi_f(t)+\phi_g(t)     =\psi_f(t) \, \e^{\mathbf i\langle t,a\rangle }+\psi_g(t) \e^{\mathbf i\langle t,b\rangle }
$$
or, equivalently, 
$$
[\psi_f(t)+\psi_f(t)][\e^{\mathbf i\langle t,c\rangle }-1]=\psi_f(t) [\e^{\mathbf i\langle t,a\rangle }-1]+\psi_g(t) [\e^{\mathbf i\langle t,b\rangle }-1]
$$
Dividing both sides by $t$ and letting $t\to 0$, we obtain $2c=a+b$, implying 
$$
\psi_f(t)[\e^{\mathbf i\langle t,c\rangle }-\e^{\mathbf i\langle t,a\rangle }]+\psi_g(t)[\e^{\mathbf i\langle t,c\rangle }-\e^{\mathbf i\langle t,b\rangle }]=0,
$$
or, 
$$
\psi_f(t)[\e^{\mathbf i\langle t,a-c\rangle }-1]+\psi_g(t)[\e^{\mathbf i\langle t,b-c\rangle }-1]=0
$$
Since $a-c=c-b$,
$$
\psi_g(t)=\psi_f(t){ \e^{\mathbf i\langle t,b-c\rangle }-1\over 1-\e^{-\mathbf i\langle t,a-c\rangle }} =\psi_f(t)\e^{\mathbf i\langle t,b-c\rangle }
$$
Starting from $f=1$, it follows that for every $f$, there is an $a(f)$ such that 
$$
\psi_f(t)=\psi_1(t)\e^{\mathbf i\langle t,a(f)\rangle }
$$
and
$$
\psi_1(t)\, \e^{i\langle t,a({1+f\over 2})\rangle }=\psi_{1+f\over 2}(t)={1\over 2} [\psi_1(t)+\psi_f(t))]={1\over 2}[1+\e^{\mathbf i\langle t,a(f\rangle }]\psi_1(t)
$$
Since $\psi_1(0)=1$ we have  for $t$ near zero
$$
\e^{\mathbf i\langle t,a({1+f\over 2})\rangle }={1\over 2}[1+\ e^{\mathbf i \langle t,a(f)\rangle }],
$$
which forces $a(f)=0$ or $a=b=c$, proving the requisite claim. 
\end{proof}

\subsection{The compactness result.} 

The metric $\mathscr D$ and the resulting topology on $\mathscr X$ provides the following compactness result. 

\begin{theorem}\label{main2}
Let $K\subset \mathscr X$ be such that as $\Theta=[\xi,\beta]$ varies over $K$, $\beta$ ranges over a uniformly tight family (in the usual weak topology) 
in ${\cM}_1(X)$. Then every sequence from $K$ has a subsequence that converges in the metric space $(\mathscr X,\mathscr D)$.
\end{theorem}

The proof of the above result needs the following notion of wide separation of measures. 
We remind the reader that for any measure $\gamma\in \Mcal_{\leq 1}(X\otimes \R^d)$, $\gamma^{\ssup 2}$ is the marginal of $\gamma$ on $\R^d$. 
Then two sequences of sub-probability measures $\gamma_n$  and $\delta_n$ on $X\otimes R^d$  are said to be {\it{widely separated}} if 
$$\lim_{n\to\infty}\int V(y_1-y_2)\gamma_n^{\ssup 2}(\d y_1)\delta_n^{\ssup 2}(\d y_2)=0$$
for some strictly positive continuous $V$ on $\R^d$ that tends to $0$ at $\infty$. This requirement is equivalent to requiring (see \cite[Lemma 2.4]{MV14})
\begin{equation}
 \lim_{n\to\infty}\int W(y_1, y_2)\gamma_n^{\ssup 2}(\d y_1)\delta_n^{\ssup 2}(\d y_2)=0 \label{decay}
\end{equation}
 for all $W\in {\cC}_2$.

\begin{lemma}\label{wide}
Let $(\gamma_n)_n$  and $(\delta_n)_n$  be widely separated sequences in $\Mcal_{\leq 1}(X\otimes \R^d)$. If $f$ is bounded and continuous on $X$ and $W\in {\cC}_k$
for some $k\ge 2$,
\begin{align}
\lim_{n\to\infty}\bigg|\int& f(x_1)\cdots f(x_k) W(y_1,\ldots, y_k)[\gamma_n+\delta_n](\d x_1, \d y_1)\cdots[\gamma_n+\delta_n](\d x_k,\d y_k)\notag\\
&-\int f(x_1)\cdots f(x_k) W(y_1,\ldots, y_k)\gamma_n(\d x_1, \d y_1)\cdots\gamma_n(\d x_k,\d y_k)\notag\\
&-\int f(x_1)\cdots f(x_k) W(y_1,\ldots, y_k)\delta_n(\d x_1, \d y_1)\cdots \delta_n(\d x_k,\d y_k)\bigg|=0
\end{align}
\end{lemma}
\begin{proof}
Let us expand the product $[\gamma_n+\delta_n](\d x_1, \d y_1)\cdots[\gamma_n+\delta_n](\d x_k,\d y_k)$. We need to show  that the cross  terms  that  involve any  pair  $\gamma_n (\d x_i, \d y_i)\cdots\delta_n(\d x_j, \d y_j)$ tend to $0$ with $n$.  For any $k$, any $W\in{\cC}_k$ and  $1\le i<j\le k$, there is a  $V\in{\cC}_2$ such that  $|W(y_1,\ldots, y_k)|\le V(y_i-y_j)$. Then from equation \eqref{decay} it follows that 
\begin{align*}
\bigg |\int  f(x_1)\cdots f(x_k) &W(y_1,\ldots, y_k)\cdots\gamma_n (\d x_i, \d y_i)\cdots\delta_n(\d x_j, \d y_j)\cdots\bigg|\\
&\le \int
V(y_i-y_j)\gamma_n ^{\ssup 2}(\d y_i)\delta_n^{(2)}(\d y_j)\to 0
\end{align*}

\end{proof}
We now turn to 
\begin{proof}[\bf{Proof of Theorem \ref{main2}}:]The proof is carried out in several steps. We will choose subsequences repeatedly and use  the diagonalization process.  We will not use multilevel subscripts but use $n$ as the index in the sequences all the time. Let $\Theta_n=[\xi_n, \beta_n]$ be given so that $(\beta_n)_n$ is uniformly tight in the weak topology in $\Mcal_1(X)$. 

\medskip\noindent
{\bf Step 1.}  By our assumption, choosing a subsequence we can assume that $\beta_n$ converges weakly to $\beta$ as probability distributions on $X$.
Assume that $\xi_n$ consists of a single orbit ${\widetilde\lambda}_n$.  Let $q_n(\ell)=\sup_{a\in \R^d}
\lambda_n^{\ssup 2}[B(a,\ell)]$ where $\lambda_n\in {\widetilde\lambda}_n$ is any measure on the orbit and $B(a,\ell)=\{y \in \R^d:|y-a|\le \ell\}$. Without loss of generality, by taking subsequences,  we can assume that $\lim_{n\to\infty}q_n(\ell)=q(\ell)$ exists for every positive integer $\ell$. Let $q=\lim_{\ell\uparrow \infty}q(\ell)$. Assume that $p=\lim_{n\to\infty}\lambda_n(X\otimes \R^d)$ exists. Then $q\le p$.

\medskip\noindent
{\bf Step 2.}  If $q=0$, and $\lambda_n^{\ssup 2}\in {\widetilde\lambda}_n$ then $\lambda_n^{\ssup 2}\vague 0$ in the vague topology and 
$$\lim_{n\to\infty}\int W(y_1,\ldots, y_k)\lambda_n^{\ssup 2}(dy_1)\cdots\lambda_n^{\ssup 2}(dy_k)=0$$
for every $W\in{\cW}_k$. From the definition of $\mathscr D$ It is now easy to check that $\mathscr D(\Theta_n,\Theta)\to 0$ where $\Theta=[\emptyset,\beta]$.

\medskip\noindent
{\bf Step 3.} If $q=p$, then $\lambda_n\star\delta_{a_n}\weak\lambda$ weakly for a suitable choice of $a_n$. Again we can verify that $\mathscr D(\Theta_n,\Theta)\to 0$ where $\Theta=[\xi,\beta]$ with $\xi$ consisting of the single orbit ${\widetilde\lambda}=\{\lambda\star\delta_a\}$, $a\in \R^d$.

\medskip\noindent
{\bf Step 4.} Assume  $0<q<p$. Choose $\ell_0$ so that $q(\ell_0) >\frac{3q}{4}$.
Pick $n_0$ large enough such  that  for $n\ge n_0$, $q_n(\ell_0)>\frac{3q}{4}$. Since $q_n(\ell)=\sup_{\lambda\in {\widetilde\lambda_n}}\lambda(B(0,\ell))$, for $n\ge n_0$,  we can find   $\lambda_n\in {\widetilde\lambda}$  such that $\lambda_n^{\ssup 2} [B(0,\ell_0)]\ge \frac{q}{2}$.  Assume, by selecting a subsequence, that $\lambda_n^{\ssup 2}$
has a vague limit $\alpha$. Clearly  $\frac{q}{2}\le \alpha(\R^d)\le q$. 

Since $\lim_{n\to\infty}\lambda_n^{\ssup 2}(B(0, r))=\alpha(B(0,r))\le \alpha(\R^d)$, for  any given $r$, we can find $n_r$ such that  for $n\ge n_r$, $\lambda_n^{\ssup 2}(B(0, r))\le \alpha(\R^d)+\frac{1}{r}$. We can assume without loss of generality that $n_r$ is increasing in $r$. 

If $\rho_n=\{\sup r: n_r\le n\}$, then $\rho_n\to\infty$ and $\lambda_n^{\ssup 2}[B(0,\rho_n)]\to \alpha(\R^d)$. The restriction of $\lambda_n$ to $B(0,\rho_n)$ will converge weakly to $\alpha$. Let us denote the restrictions of  $\lambda_n$ to $X\otimes B(0,\rho_n)$ and its complement $X\otimes [B(0,\rho_n)]^c$ by $\gamma_{1,n}$ and $\delta_{1,n}$ respectively, so that $\lambda_n=\gamma_{1,n}+\delta_{1,n}$. Then a suitable subsequence of $\gamma_{1,n}$ will have a weak limit $\gamma_1$ on $X\otimes \R^d$ and $\lim_{n\to\infty}\int V(y_1-y_2)\delta_{1,n}(\d y_1)\gamma_{1,n}(\d y_2)=0$.

The orbit ${\widetilde\gamma}_1$ of $\gamma_1$  is a member of the set $\xi$. We work with $\delta_{1,n}$ and extract in the same way a limit ${\widetilde\gamma}_2$ and a leftover piece $\delta_{2,n}$. The process ends when $q=0$ at some stage and we end up with a finite set $\xi$ of $\{{\widetilde\gamma}_j\}$. If we do not end at a finite stage we end up with a countable set. In either case $\sum_{{\widetilde\gamma}_j\in \xi}\gamma_j^{\ssup 1}\le \beta$.

\medskip\noindent
{\bf Step 5.}  In estimating the distances $\mathbf D^\ast(\xi_1, \xi_2)$ if  either  $\xi_1$ or $\xi_2$ consists of several orbits $\{{\widetilde\lambda}\}$, we can afford to  ignore orbits of measures with  total mass at most $\eps$. Since $k(r)\ge 2$, their total contribution to $\mathbf D^\star$ is at most

$$
\sum_{r\ge 2}\sum_{{\widetilde \lambda\colon}\atop  { {\widetilde \lambda}(X\otimes \R^d)\le \eps}}
\frac{[{\widetilde \lambda}(X\otimes \R^d)   ]^{k(r)}}{2^r}\le\sum_{ {{\widetilde \lambda\colon}\atop  { {\widetilde \lambda}(X\otimes \R^d)\le \eps}}}  [{\widetilde \lambda}(X\otimes \R^d)   ]^2\le \eps \sum_{ {{\widetilde \lambda\colon}\atop  { {\widetilde \lambda}(X\otimes \R^d)\le \eps}}}  {\widetilde \lambda}(X\otimes \R^d)   \le \eps
$$

We need to examine  at most $\eps^{-1}$ orbits  in each $\xi_n$. Taking a subsequence we can assume that  the   number of such  orbits $m$  is  the same  in each $\xi_n$ and link them as ${\widetilde\lambda}_{i,n}$ and deal with each sequence on its own. As limit along subsequences they will each generate a  collection of orbits and their union, denoted by $\xi_\eps$, will again be a collection of orbits.  In the end we can let $\eps\to 0$ 
so that $\xi_\eps$ increases to a limiting collection of orbits $\xi$. Since we have uniform control over the combined  contributions of all orbits with small masses, we can now pass to the limit  
$\eps\to 0$ . Lemma \ref{wide} will now allow us to show that we have convergence in the metric $\mathscr D$. 
From now on the details are identical to the proof of 
\cite[Theorem 3.2]{MV14} except we now have $X\otimes \R^d$  instead of $\R^d$. We omit the details to avoid repetition. 

\end{proof}

The following result is an immediate consequence of Theorem \ref{main2}.
\begin{cor}\label{cor:main2}
Let $\{\Pi_n\}$ be a sequence of probability distributions on $\mathscr X$. For any $\Theta=[\xi,\beta] \ in \mathscr X$ if 
the distributions $\{\widehat{\Pi}_n\}$  of $\beta$ under $\Pi_n$ is uniformly  tight as a subset of ${\cM}_1(X)$ then  so is $\{\Pi_n\}$ as a subset of  ${\cM}_1(\mathscr X)$. 
\end{cor}
\begin{remark}\label{rmk:main2}
From the definition (\ref{metric1}) and (\ref{metric3}) of the metric $\mathscr D$ on $\mathscr X$ it follows that for any continuous function $V$ on $\R^d$ that  tends to $0$  at $\infty$ the  function 
$$\Psi(V, \Theta)=\sum_{\widetilde\lambda\in \xi}\int V(y_1-y_2){\widetilde\lambda}^{\ssup 2}(\d y_1){\widetilde\lambda}^{\ssup 2}(\d y_2)$$
is a bounded continuous function of $\Theta\in \mathscr X$, where $\Theta=[\xi,\beta]$, $\xi=\{\widetilde \lambda_j\}$ and $\widetilde\lambda^{\ssup 2}$ is the marignal of any $\widetilde\lambda\in \xi$ on $\R^d$.\qed
\end{remark}

\section{Stationary Processes and Processes with Stationary Increments}\label{sec-stationarity}

\subsection{Some notation.}

We start with the space $\Omega=\{\omega: (-\infty, \infty)\to \R^d: \omega(\cdot) \,\,\mathrm{continuous}\}$ of $\R^d$-valued continuous functions on $(-\infty,\infty)$, which, 
equipped with the topology of uniform convergence on bounded intervals, 
 is a complete separable metric space. The Borel $\sigma$-field of $\Omega$, denoted by $\cF$, is generated by $\{\omega(t): -\infty<t<\infty\}$.
If $-\infty<a<b<\infty$ the $\sigma$-field  ${\cF}_{[a,b]}$ is the one generated by the increments $\{\omega(t)-\omega(s)\}$, $a\le s<t\le b$,
and ${\cF}_{\mathrm{inc}}=\sigma(\cup_{-\infty<a<b<\infty}{\cF}_{[a,b]})$ is the $\sigma$-field of increments of $\omega$.
If we denote by
 $$
 \Omega_0=\{\omega\in \Omega \in\Omega: \omega(0)=0\}\subset \Omega, 
 $$
then $\Omega_0$ can be identified (via a one-to-one map) with the {\it{space of increments}}, i.e. continuous functions 
$$
h(s,t): \R\times \R\to \R^d\qquad \mbox{such that}\quad h(s,t)+h(t,u)=h(s,u).
$$
Alternatively, we can also view the space of such functions $h(\cdot,\cdot)$ as equivalence classes $\Omega/\sim$ modulo constants, i.e., 
for $\omega,\omega^\prime \in\Omega$, we declare $\omega\sim \omega^\prime$ if for some constant $a\in \R^d$, $\omega(t)=\omega^\prime(t)+a$ for all $t\in \R$. 
Topology on $\Omega_0$ is the natural one of uniform convergence on bounded subsets. 

Next we remark that we can identify  $\Omega $ with $\Omega_0\otimes \R^d$ by  mapping  
$$
\Omega\ni \omega\leftrightarrow (\omega^\prime,a) \qquad\mbox{where}\qquad a=\omega(0),\,\,\omega^\prime(t)=\omega(t)-\omega(0).
$$
Thus, a sub-probability measure $\Q$ on $\Omega$ can be viewed as a measure on $\Omega_0\times \R^d$ and  will then have marginals ${\Q}^{\ssup 1} \in \Mcal_{\leq 1}(\Omega_0)$, 
${\Q}^{\ssup 2} \in \Mcal_{\leq 1}(\R^d)$, respectively. Note that the marginal ${\Q}^{\ssup 1}$ is just the distribution of the increments of a process that has  $\Q$ for its distribution on $\cF$,

\subsection{Stationary cocycles and group actions of $(S_t)_{t\in\R}$ on $\Omega$ and $(\theta_t)_{t\in\R}$ on $\Omega_0\otimes \R^d$.}

We have the group of time translations  $(S_t)_{t\in \R}$ acting on $\Omega$ as
\begin{equation}\label{eq:St}
S_t\colon \Omega\to \Omega\qquad (S_t \omega)(s)=\omega(s+t)
\end{equation}
while the group $(\theta_t)_{t\in \R}$ acts on $\Omega_0\times R^d$ as well as $\Omega_0$ alone as
\begin{equation}\label{eq:thetat}
\begin{aligned}
&\theta_t:\Omega_0\otimes \R^d\to \Omega_0\otimes \R^d\quad \mbox{with}\,\,(\theta_t)(\omega,x)=(\omega_t, x_t) \\
&\mbox{where}\quad \omega_t(s)=\omega(t+s)-\omega(t),\,\,\mbox{and}\,\,\,x_t=x+\omega(t) \quad\mbox{and} \\
&\theta_t:\Omega_0\to \Omega_0 \qquad \mbox{with}\,\,(\theta_t\omega)(s)=\omega_t(s).
\end{aligned}
\end{equation}
Note that $S_t$-invariant probability measures on $\Omega$ are precisely $\theta_t$-invariant probability measures on $\Omega_0\otimes \R^d$,
and are called {\it{stationary processes}} and are denoted by $\Mcal_{\mathrm{s}}(\Omega_0\otimes\R^d)$. On the other hand, 
$\theta_t$-invariant probability measures on $\Omega_0$ are {\it{processes with stationary increments}}, and are denoted by $\Mcal_{\mathrm{si}}(\Omega_0)$.  
We can have a probability measure on $\Omega_0\otimes \R^d$ whose marginal on $\Omega_0$ is $\theta_t$-invariant but it is not $\theta_t$-invariant on $\Omega_0\times R^d$. 
These are precisely {\it{non-stationary processes with stationary increments}}.  The following lemma provides a useful  
criterion that determines if a process $\beta$ with stationary increments is the process of increments of a stationary process $\Q$. 
We will see that even if it is, it is not unique, and in fact,  it is the {\it{entire orbit}} $\{\Q\star\delta_a\colon \, a\in \R^d\}$ where $\Q\star\delta_a$ is given by
$(\Q\star\delta_a)[A\otimes B]=\Q[A\otimes (B-a)]$on $\Omega_0\otimes \R^d$ and by $(\Q\star\delta_a)[A]=\Q[A-a]$ on $\Omega$.

\begin{theorem}\label{lemma:cocycle}
Let $\beta$ be an ergodic process with stationary increments (i.e., $\beta\in \Mcal_{\mathrm{si}}(\Omega_0)$ is a $\theta_t$-invariant and ergodic probability distribution on  $\Omega_0$). Then either
\begin{equation}\label{eq:cocycle}
\lim_{\eps\to 0}\E^{\beta}\bigg[\eps \int_0^\infty \e^{-\eps t} V(\omega(t)-\omega(0)) \d t\bigg]=0
\end{equation}
for all continuous  functions $V: \R^d\to \R$ with $\lim_{|x|\to\infty} |V(x)|=0$
or there is a $\theta_t$-invariant probability distribution $\Q\in \Mcal_{\mathrm s}(\Omega_0\otimes\R^d)$ such that ${\beta}={\Q}^{\ssup 1}$, i.e. $\Q$ is a stationary process on $\Omega$ with 
$\beta$ being the distribution of its increments.
\end{theorem}
\begin{remark}
For the proof of the above result it is enough to check the condition \eqref{eq:cocycle} for one strictly positive $V$.
\end{remark}
\begin{proof}[\bf{Proof of Theorem \ref{lemma:cocycle}}:] 
Define ${\Q}_0 (\d \omega \d y)={\beta}(\d \omega)\otimes \delta_0(\d y) \in \Mcal_1(\Omega_0\otimes \R^d)$ and set
\begin{equation}\label{eq:Lambda}
 {\nu}_\eps (\d \omega\d y)=\eps\int_0^\infty \d t \, \e^{-\eps t}\,\, ({\Q}_0 \theta_t^{-1})(\d\omega \d y)
\end{equation}
Since the marginal ${\nu}_\eps^{\ssup 1}$  on $\Omega_0$ of ${\nu}_\eps$ is $\beta$ for every $\eps>0$, the family $\{\nu_\eps^{\ssup 1}\}$ is (weakly) uniformly tight on $\Omega_0$ 
and by Lemma \ref{lemma-weague}, $\{\nu_\eps\}$ has wea-gue limit points. Let $ \Q $ be any nonzero  {\it weague} limit point of ${\nu}_\eps$ as $\eps\to 0$. 
We will show that for any $\tau>0$,  $\theta_\tau \Q= {\Q}$. Then  $\Q$ on $\Omega$   will be stationary and its marginal ${\Q}^{\ssup 1}$ on $\Omega_0$ will  be dominated by  $\beta$.  

To show that $\theta_\tau\Q=\Q$ it is enough to verify  that for any $\tau>0$ and any continuous  $F:\Omega_0\otimes \R^d\to \R$ satisfying 
\begin{equation}\label{eq:F}
\lim_{|x|\to\infty}\sup_{\omega\in \Omega_0}|F(\omega,x)|=0
\end{equation}
we have
\begin{equation}\label{E3.1}
\int_{\Omega_0}\int_{\R^d} F( \theta_\tau \omega, y+\omega(\tau))  {\Q}(\d \omega,\d y)=\int_{\Omega_0}\int_{\R^d}F( \omega, y) {\Q}(\d \omega \d y)
\end{equation}
Actually it is sufficient to prove the claim by taking $F$ of the form $G(\omega)H(y)$ where $G$ is bounded and continuous on $\Omega_0$ and $H$ is continuous and compactly supported on $\R^d$.
 
By construction, we note that for any continuous and bounded function $F$ on $\Omega_0\otimes \R^d$, 
$$
\begin{aligned}
&\bigg| \int_0^\infty F\big(\theta_s(\cdot)\big) \,\,\eps \e^{-\eps s}\,\, \d s - \int_0^\infty F\big(\theta_{\tau+s}(\cdot)\big) \,\,\eps \e^{-\eps s}\,\, \d s \bigg| \\
&=\bigg| \int_0^\tau F\big(\tau_s(\cdot)\big) \,\,\eps \e^{-\eps s}\,\, \d s - \int_t^\infty F\big(\tau_{s}(\cdot)\big) \,\,\big[\eps \e^{-\eps (s-\tau)}- \eps \e^{-\eps s}\,\, \d s \bigg|  \\
&\leq \eps \tau \e^{\eps\tau}\|F\|_\infty+ \|F\|_\infty [\e^{\eps \tau}-1 ],
\end{aligned}
$$
which, together with the definition \eqref{eq:Lambda}, imply that $\| {\nu}_\eps- \theta_\tau  {\nu}_\eps\|\le  |\e^{\tau \eps}-1|+\tau\eps \e^{\tau \eps}$. 
Since $ {\nu}_\eps $ converges  wea-guely ( along a subsequence)   to $ \Q$, to show that $\theta_\tau\Q=\Q$, 
it suffices to check that $ {\nu}_\eps \theta_\tau^{-1} \to {\Q}\theta_\tau^{-1}$ wea-guely, which is equivalent to showing
\begin{equation}\label{E3.2}
\int_{\Omega_0}\int_{\R^d} F(\theta_\tau \omega, y+\omega(\tau))  {\nu}_\eps (\d \xi, \d y) \to\int_{\Omega_0}\int_{\R^d} F(\theta_\tau \omega, y+\omega(\tau)) {\Q}(\d \omega, \d y)
\end{equation}
for any continuous  $F:\Omega_0\times \R^d\to \R$ satisfying \eqref{eq:F}. The distribution of  $\omega(\tau)$ under ${\nu}_\eps$ is that of the original increment $\omega(\tau)-\omega(0)$  under $\beta$ and so  given $\delta>0$, there is a function $g(z)$ with compact support in $\R^d$ with $0\le g\le 1$ such that  for all $\eps>0$
\begin{equation}\label{E3.3}
\int (1-g(\omega(\tau)))  {\nu}_\eps^{\ssup 1} (\d\omega)=\int (1-g(\omega(\tau)))\beta(\d \omega)\le \delta
\end{equation} 
and since ${\Q}^{\ssup 1}\le \beta$ we have  $\int (1-g(\omega(\tau)))  {\Q}^{\ssup 1}(\d \omega)\le \delta$.
If  we replace $F(\theta_\tau \omega, y+\omega(\tau)) $ by $G(\omega,a)=g(\omega(\tau)) F(\theta_\tau \omega, y+\omega(\tau))$,  since $\omega(\tau)$ is now  restricted to a bounded set,
we have
\begin{equation*}
\lim_{|y|\to\infty}\sup_\omega |G(\omega,y)|=0
\end{equation*}
and (\ref{E3.2}) follows with the help of (\ref{E3.1}) and (\ref{E3.3}). Thus $\theta_\tau\Q=\Q$ for all $\tau>0$. 

We will now conclude. Note that 
the measures $ {\nu}_\eps $ all have marginals ${\nu}_\eps^{\ssup 1}=\beta$  and the wea-gue limit $\Q $ of $\nu_\eps$ is $\theta_t$-invariant and ${\Q}^{\ssup 1}$  is  dominated by $\beta$. 
Since $\beta$ is  ergodic it is equal to $c\beta$ for some $0\le c\le 1$ and ${1\over c}{\Q}$ is a stationary process with $\beta={1\over c}{\Q}^{\ssup 1}$ provided we can show  that $c> 0$. However if $c=0$, we have 
\begin{equation*}
\lim_{\eps\to 0}\int_{\Omega_0}\int_{\R^d} F(\omega, y) {\nu}_\eps(\d \omega, \d y)=0
\end{equation*}
for continuous functions $F$ satisfying \eqref{eq:F}. In particular $F(\omega,y)$ can be any continuous function $V(y)$ satisfying 
$\lim_{|y|\to\infty}|V(y)|=0$ and hence, 

\begin{equation*}
\begin{aligned}
\int_{\Omega_0}\int_{\R^d} V(y){\nu}_\eps(\d \omega, \d y)&=\int_{\R^d} V(y){\nu}_\eps^{\ssup 2}( \d y)\\
&=\E^\beta\bigg[\eps\int_0^\infty \e^{-\eps t} V(\omega(t)-\omega(0)) \d t\bigg]\to 0
\end{aligned}
\end{equation*}
as $\eps\to 0$. 
\end{proof}

In order to apply Theorem \ref{lemma:cocycle} in the present context, we need another fact, whose proof can be found in \cite[Theorem 4.1]{DV83}.
\begin{lemma}\label{flow}
Let $\mathcal Y$ be a Polish space and $(S_t)_{t\geq 0}$  a one-parameter family of homeomorphisms of $\mathcal Y$. For each $\mathrm y\in \mathcal Y$
let 
\begin{equation*}
\nu_{\eps}(\mathrm y,\cdot)=\eps\int_0^\infty \e^{-\eps t} \delta_{S_t(\mathrm y)}  \d t,
\end{equation*}
i.e., for any $\mathrm y\in \mathcal Y$ and $A\subset \mathcal Y$, $\nu_\eps(\mathrm y,A)= \int_0^\infty \d t \, \eps \e^{-\eps t} \1\{S_t(\mathrm y) \in A\}$. Then the following implications hold.
\begin{itemize}
\item For each $\eps>0$, $\mathrm y\to \nu_{\eps}(\mathrm y,\cdot)$ is a map from $X\to {\mathcal M}_1(\mathcal Y)$ and any weak limit point $\nu=\lim_{\eps\to 0} \nu_\eps$  
is invariant under $S_t$. 
\item Let ${\mu}_n$ be a sequence of  $S_t$-invariant probability 
measures on $\mathcal Y$ converging weakly to $\mu$,  and ${\lambda}_{\eps,\mu_n}$ be the distribution  of 
$\nu_{\eps}(\mathrm y, \cdot)$ on ${\mathcal M}_1(\mathcal Y)$, i.e., 
$$
\lambda_{\eps,\mu_n}(G)= \mu_n\big[\mathrm y \in \mathcal Y\colon \nu_\eps(y,\cdot)\in G\big] \qquad\forall \,\, G\subset \Mcal_1(\mathcal Y).
$$
Then $(\lambda_{\eps,\mu_n})_n$ is uniformly tight on ${\mathcal M}_1(\mathcal Y)$ and for any sequence $\eps_n\to 0$, any weak limit point $\lambda$ of 
${\lambda}_{\eps_n,\mu_n}$ is supported on the space ${\mathcal M}_{\mathrm{inv}}(\mathcal Y)$ of $S_t$-invariant probability measures on $\mathcal Y$. Moreover, 
\begin{equation*}
\int_{{\mathcal M}_{\mathrm{inv}}(\mathcal Y)} \nu\  \lambda(\d \nu)=\mu.\qed
\end{equation*}
\end{itemize}
\end{lemma}

The following corollary is an immediate consequence of Lemma \ref{flow} for the particular case $\mathcal Y=\Omega_0\otimes \R^d$.

\begin{cor}\label{cor:cocycle}
Let ${\Q}_n$ be a  sequence of invariant probability measures on $\Omega_0\otimes \R^d$ such that
their marginals ${\Q}_n^{\ssup 1}\in \Mcal_1(\Omega_0)$ converge weakly to a limit $\beta\in \Mcal_1(\Omega_0)$. For any $\omega\in\Omega_0$ and $\eps>0$ let
\begin{equation}\label{eq:nu}
\nu_\eps(\omega, \cdot)= \nu_\eps\big((\omega,0),\cdot\big)=\eps\int_0^\infty \e^{-\eps t}\delta_{\theta_t(\omega,0)} \d t\in {\cM}_1(\Omega_0\otimes \R^d)
\end{equation}
Let $\eps_n$ be any sequence such that $\eps_n\to 0$ and  $\Pi_n\in {\cM}_1({\cM}_1(\Omega_0\times \R^d))$ is the distribution of $\nu_{\eps_n}$ under $\Q_n$, i.e., $(\omega,0)\in \Omega_0\otimes \R^d$
is sampled according to $\Q_n$ and for any $\mathscr A\subset \Mcal_1(\Omega_0\otimes\R^d)$ we set
\begin{equation}\label{eq:Pi}
\Pi_n(\mathscr A)\stackrel{(def)}= \Q_n \, \big[\nu_{\eps_n}\big((\omega,0),\cdot\big)\in \mathscr A\big].
\end{equation}
Then the projection  ${\widehat\Pi}_n\in {\cM}_1({\cM}_1(\Omega_0))$ of $\Pi_n$ is uniformly tight  and any weak limit point $\widehat\Pi=\lim_{n\to\infty} \widehat\Pi_n$  
is supported on the space of $\theta_t$-invariant (but not necessarily  ergodic) distributions ${\cM}_{\mathrm{si}}(\Omega_0)$ and moreover, 
\begin{equation*}
\lim_{n\to\infty} \Q^{\ssup 1}_n= \beta=\int_{{\cM}_{\mathrm{si}}} \nu\, {\widehat\Pi}(\d \nu).
\end{equation*}
\end{cor}
\qed
\begin{remark}Note that the space ${\cM}_{\mathrm{si}}(\Omega_0)$, the space of processes with stationary increments consists of two sets ${\cM}_{\mathrm{coc}}$ and ${\cM}_{\mathrm{ncoc}}$ corresponding to those that are increments of a stationary process and those that are not.
Thus Corollary \ref{cor:cocycle} implies that the limit $\beta$ (in Corollary \ref{cor:cocycle}) can be written as the sum $\beta=\beta_{\mathrm{coc}}+\beta_{\mathrm{ncoc}}$ where
\begin{equation*}
\beta_{\mathrm{coc}}=\int_{{\cM}_{\mathrm{coc} }} \nu\, \Pi(\d \nu)\ \ {\rm and}\ \ \beta_{\mathrm{ncoc}}=\int_{{\cM}_{\mathrm{ncoc}}} \nu\,\Pi(\d \nu)
\end{equation*}
For any $\nu\in {\cM}_{\mathrm{coc}}$ we have a stationary distribution ${\widetilde\alpha}_\nu=\{\alpha_\nu\star \delta_a\colon a\in \R^d\}$ that are the possible marginals. 
\end{remark}
\qed

\section{Relative entropy estimates and variational arguments.}\label{sec-entropy}

\subsection{Relative entropy of processes with stationary increments.}

From now on we shall assume that $d=3$. Let us recall  that for any $a<b$,  we  denote by $\mathcal F_{[a,b]}$  the 
$\sigma$-algebra  generated by all increments $\{\omega(s)-\omega(r)\colon a \leq r < s \leq b\}$ and $\P$ refers to the law of three-dimensional Brownian increments.  
For any process $\Q\in \Mcal_{\mathrm{si}}(\Omega_0)$  with stationary increments, if $\Q_{0,\omega}$ denotes the regular conditional probability distribution  of $\Q$ given $\mathcal F_{(-\infty,0)}$, which is
 the $\sigma-$field generated by $\{(\omega(t)-\omega(s)): -\infty<s<t\le 0\}$ then 
\begin{equation}\label{eq1:Ht}
H_t(\Q| \P)= \E^{\Q}\bigg[h_{\mathcal F_{[0,t]}}(\Q_{0,\omega} \big | \P)\bigg]
\end{equation}
defines the entropy of the increments of the process $\Q$ with respect to $\P$ over the $\sigma$-field  ${\cF}_{[0,t]}$. Here for
 any two probability measures $\mu$ and $\nu$ on any $\sigma$-algebra of the form ${\cF}=\mathcal F_{[a,b]}$ on $\Omega$, we denote by 
\begin{equation}\label{eq2:Ht}
h_{\mathcal F}(\mu|\nu)=\sup_f \bigg\{\int f \d\mu- \log\bigg(\int \e^f \d\nu\bigg)\bigg\}
\end{equation}
the relative entropy of the probability measure $\mu$ with respect to $\nu$ on $\mathcal F$ and the supremum above is taken over all 
continuous, bounded and $\mathcal F$-measurable functions. For our purposes,  
it is useful to collect some well-known properties of $H_t(\mathbb Q| \P)$. 
\begin{lemma}\label{lemma-entropy}
Either $H_t(\mathbb Q| \P)\equiv\infty$ for all $t>0$, or, $H_t(\mathbb Q|\P)= t H(\mathbb Q|\P)$, where the map  
$H(\cdot| \P): \Mcal_{\mathrm{si}} \to [0,\infty]$, called the {\it ( specific) relative entropy of a process  $\Q$ with stationary increments} satisfies the following properties: 
\begin{itemize}
\item $H(\cdot|\P)$ is convex and lower-semicontinuous in the usual weak topology. 
\item$H(\cdot|\P)$ is coercive (i.e., for any $\ell\geq 0$, 
the sub-level sets $\{\mathbb Q\colon H(\mathbb Q| \P)\leq \ell\}$ are weakly compact as measures on $\Omega_0$.
\item The map $\mathbb Q\mapsto H(\mathbb Q|\P)$ is in fact linear. In particular, for any probability measure $\Gamma$ on $\Mcal_{\mathrm{si}}$, 
\begin{equation}\label{linear}
H\bigg(\int \mathbb Q \, \Gamma(\d\mathbb Q)\bigg|\P\bigg)= \int H(\mathbb Q|\P) \,\Gamma(\d\mathbb Q).
\end{equation}
\end{itemize}
\end{lemma}
\begin{proof}
We refer to \cite[Section 3]{DV4} for the proofs of the above assertions. Indeed, the fact that either $H_t(\mathbb Q| \P)\equiv\infty$ for all $t>0$, or, $H_t(\mathbb Q|\P)= t H(\mathbb Q|\P)$
can be found in \cite[Theorem 3.1]{DV4}. The convexity and lower semicontinuity of the $H(\cdot| \P)$ is proved in  \cite[Theorem 3.3]{DV4}, while coercivity of this map 
is a result of the variational representation of $H_t(\cdot|\mathbb P)$ proved in \cite[Theorem 3.2]{DV4}. Finally, the linearity of $H(\cdot|\P)$ is shown in \cite[Theorem 3.5]{DV4}.
\end{proof}
\medskip

Finally we remark that if $\Q\in \Mcal_{\mathrm s}(\Omega_0\otimes \R^d)$ is a stationary process, we can also consider the $\sigma-$ fields $\Sigma_{[a,b]}$ generated by $\{\omega(t): a\le t\le b\}$ and define
$$
{\widehat H}_t(\Q| \P)= \E^{\Q}\bigg[h_{ \Sigma_{[0,t]}}(\Q_{0,\omega} \big | \P)\bigg]
$$
where $\Q_{0,\omega}$ is now the conditional probability given $\Sigma_{(-\infty,0]}$. Since a stationary process is also one with stationary increments both $H_t(\Q| \P)$ (defined in \eqref{eq1:Ht}) and ${\widehat H}_t(\Q| \P)$ make sense for $\Q\in {\cM}_{s}$ and the same statements as in Lemma \ref{lemma-entropy} continue to hold in this case too, and in fact it is not difficult to see that these two objects coincide, i.e., 
\begin{lemma}
For any $\Q\in \Mcal_{\mathrm s}(\Omega_0\otimes \R^d)$, 
\begin{equation}\label{equality-entropy}
H(\Q|\P)  =\widehat H(\Q|\P)
\end{equation}
\end{lemma}
\begin{proof}
Indeed, since by \eqref{linear} in Lemma \ref{lemma-entropy} both $H$ and $\widehat H$ are linear we can assume that $\Q$ is stationary and ergodic (else $\Q$ can be written as a mixture 
of stationary ergodic measures). If either $H(\Q)$ or ${\widehat  H}(\Q)$ is finite, then \eqref{eq1:Ht}-\eqref{eq2:Ht} imply that $\E^{\Q}[|\omega(1)-\omega(0)|]<\infty$. By the ergodic theorem, we have 
$$
\omega(0)=\lim_{T\to\infty}\frac{1}{T}\int_0^T (\omega(0)-\omega(t)) dt+E^{\Q}[\omega(0)]
$$
and since there is no essential difference between the two $\sigma-$fields $\mathcal F$ and $\Sigma$, we have the desired equality \eqref{equality-entropy}. 
\end{proof}

\subsection{Relative entropy estimates.}
The following lemma allows us to apply the results from Section \ref{notation} and Section \ref{sec-stationarity} to the singular function $x\mapsto 1/|x|$ in $\R^3$. 
\begin{lemma}\label{lemma-Coulomb1}
For any $\eta>0$, let 
\begin{equation}\label{Y_eps}
\begin{aligned}
&V(x)=\frac 1 {|x|} \qquad \,\,\, V_\eta(x)= \frac 1 {(\eta^2+|x|^2)^{1/2}} \quad \mbox{and}\\
&Y_\eta(x)= (V-V_\eta)(x)
\end{aligned}
\end{equation}
Then  
\begin{equation}\label{E4.3}
\lim_{\eta\to 0} \,\,\sup_{\mathbb Q\in \Mcal_{\mathrm{si}}\atop H(\mathbb Q|\mathbb {\mathbb P})\le C} \,\,\sup_{0< \eps\le 1} \,\,\eps\,\,\E^{\mathbb Q} \bigg[\int_0^\infty \d t \,\, \e^{-\eps t}\,\,Y_\eta\big(\omega(t)- \omega(0)\big)\bigg]=0.
\end{equation}
where in the supremum above $C<\infty$ is any finite constant. 
Moreover, given any $\delta>0$, there is a constant $C(\delta)$ such that for all $\eps\le 1$ and $\Q\in{\cM}_{\mathrm{si}}(\Omega_0)$ with $H(\Q|\P)<\infty$, 
\begin{equation}\label{E4.2}
\E^{\mathbb {\mathbb Q}} \bigg[\eps\int_0^\infty  \,\,  \frac{\e^{-\eps t}}{|\omega(t)-\omega(0)|} \d t\bigg]\le \delta H({\mathbb Q}|{\mathbb P})+C(\delta).
\end{equation}

\end{lemma}
\begin{proof}
With any $\rho>0$, we estimate: 
\begin{equation}\label{est:entropy}
\begin{aligned}
F(t)&\stackrel{\mathrm{(def)}}{=}\E^{\Q}\bigg[\int_0^t  Y_\eta(\omega(t)-\omega(0)) \d t \bigg]\\
&\le\frac{t}{\rho}H({\Q}|{\P})+\frac{1}{\rho}\log \E^{\P}\bigg[\exp \bigg(\rho\int_0^tY_\eta(\omega(s)-\omega(0))\d s\bigg)\bigg]\\ 
&\le\frac{t}{\rho}H({\Q}|{\P})+\frac{1}{\rho}\log \E^{\P}\bigg[\exp \bigg(\rho\int_0^\infty Y_\eta(\omega(s)-\omega(0))\d s\bigg)\bigg]\\ 
&\le \frac{t}{\rho}H({\Q}|{\P})+\frac{1}{\rho}\log \bigg[\frac 1 {1-\rho \sup_{x\in \R^3} \E^{{\P}}\big[\int_0^\infty Y_\eta(x+\omega(s)-\omega(0))\d s\big]}\bigg]\\
&=\frac{t}{\rho}H({\Q}|{\P})-\frac{1}{\rho}\log [1-\rho h(\eta)]
\end{aligned}
\end{equation}
In the first upper bound in the above display we have used the relative entropy inequality,
$$
\E^\Q\bigg[\int_0^t V(\omega(s)-\omega(0))\d s\bigg ]\le \frac{t \,H(\Q|\P)}{\rho}+\frac{1}{\rho}\log \E^\P\bigg[\exp\bigg(\rho\int_0^tV(\omega(s)-\omega(0))\d s\bigg)\bigg]
$$
 and in the third upper bound we have used Khasminski's lemma  stating that for $V\ge 0$,
 $$
\E^{\P}\bigg[\exp \bigg(\rho\int_0^\infty V(\omega(s)-\omega(0))\d s\bigg)\bigg]\le \frac 1 {1-\rho\sup_{x\in \R^3} \E^{\P}\big[\int_0^\infty V(x+\omega(s)-\omega(0))\d s\big]}.
 $$
  It is not hard to check that the supremum in
 $$ 
 \sup_{x\in \R^3} \E^{\P}\big[\int_0^\infty \frac{1}{|x+\omega(s)-\omega(0)|^\frac{3}{2}}ds\big]
 $$
 is attained at $x=0$. Furthermore with \eqref{Y_eps} we can estimate $0\le Y_\eta(x)\le \frac {c{\sqrt \eta}}{ |x|^\frac{3}{2}}$, and therefore 
 $$
 \begin{aligned}
h(\eta) =E^{\P}[\int_0^\infty Y_\eta(\omega(s)-\omega(0))ds]&\le c{\sqrt\eta}\int_{R^3}\int_0^\infty \frac{1}{|x|^\frac{3}{2}}
\frac{1}{(2\pi t)^\frac{3}{2}}e^{-\frac{|x|^2}{2t}} \d t \d x\\
&=c^\prime {\sqrt\eta}\to 0
\end{aligned}
 $$
as   $\eta\to 0$.  For 
$\eps\le 1$ and $H({\mathbb Q}|{\mathbb P})\le C$, the expectation  on the left hand side in \eqref{E4.3} can be estimated, with the help of \eqref{est:entropy} and repeated 
integration by parts, which yields 
\begin{align*}
\E^{\mathbb Q}\bigg[\eps\int_0^\infty \d t \,\, \e^{-\eps t}\,\,Y_\eta\big(\omega(t)- \omega(0)\big)\bigg]
&=\eps\int_0^\infty \,\, \e^{-\eps t}\,\,F'(t)\, \d t\\
&=\eps^2\int_0^\infty e^{-\eps t} F(t) \d t\\
&\le \frac{1}{\rho} H({\mathbb Q}|{\mathbb P})-\frac{\eps}{\rho}\log  [1-\rho h(\eta)]\\
&\le \frac{C}{\rho}-\frac{1}{\rho}\log  [1-\rho h(\eta)].
\end{align*}
We can now let $\eta\to 0$ followed by $\rho\to\infty$ to deduce \eqref{E4.3}. 
To prove \eqref{E4.2}, given any $\delta>0$ in the previous step we can take $\eta$ to be small enough so that $h(\eta)<\frac{\delta}{2}$. Then with $\rho=\delta^{-1}$ and $V_\eta(x)\le \frac{1}{\eta}$, we have
\begin{align*}
\E^{\mathbb Q} \bigg[\eps\int_0^\infty  \,\,  \frac{\e^{-\eps t}}{|\omega(t)-\omega(0)|} \d t\bigg]&
\le \delta H({\Q}|{\P})+ \frac{1}{\eta}+\delta\log 2\\
&= \delta H({\Q}|{\P})+C(\delta)
\end{align*}
which proves \eqref{E4.2}.
\end{proof}
Combining the above result with Theorem \ref{lemma:cocycle} we now have
\begin{cor}\label{cor:cocycle}
Let $\Q$ be an ergodic process with stationary increments such that $H(\Q|\P)<\infty$. Then either
\begin{equation}\label{eq:cocycle}
\lim_{\eps\to 0}\E^{\Q}\bigg[\eps \int_0^\infty \frac{\e^{-\eps t}}{|\omega(t)-\omega(0)|} \d t\bigg]=0
\end{equation}
or there is a stationary process on $\Omega$ with 
$\Q$ being the distribution of its increments. \qed
\end{cor}

\begin{lemma}\label{lemma:entropy}
If $\Q \in \Mcal_{\mathrm s}(\Omega_0\times \R^3)$  is a stationary process with marginal $\mu\in \Mcal_1(\R^3)$ and $H({\Q}|\P)<\infty $, then 
there exists a non-negative function $\psi\in H^1(\R^3)$ with $\|\psi\|_2=1$ such that
$\d\mu=[\psi (x)]^2 \d x$. Moreover, if $\psi$ is strictly positive and $\log \psi$ is twice continuously differentiable with bounded derivatives, we also have

\begin{equation}\label{E4.4}
H({\Q}|\P)=\frac{1}{2}\int_{\R^3} |(\nabla\psi)(x)|^2 \d x + H(\Q| \widetilde\Q_\psi)
\end{equation}
where ${\widetilde\Q}_\psi\in \Mcal_1(\Omega_0\otimes \R^3)$, the stationary Markov process with generator $\frac{1}{2}\Delta+ \frac{\nabla \psi}{\psi} \cdot \nabla$  and marginal $\psi^2(x) \d x\in \Mcal_1(\R^3)$.
\end{lemma}

\begin{remark}
The notation $\widetilde\Q_\psi$ for the process with generator $\frac{1}{2}\Delta+ \frac{\nabla \psi}{\psi} \cdot \nabla$ will be justified later, see Theorem \ref{thmain} and Remark \ref{rmk:thmain} below. \qed
\end{remark}

\begin{proof}[{\bf Proof of Lemma \ref{lemma:entropy}}]
Note that the first statement of the lemma is immediate. Assume that $\psi$ is strictly positive and $\log\psi$ has the required regularity. Let ${\widetilde\Q}^\omega_\psi $ and ${\Q}^\omega$ be respectively the conditional distributions of $\widetilde\Q_\psi$ and $\Q$  on ${\mathcal F}_{[0,1]}$  given ${\cF}_{(-\infty,0)}$. If we write 
$$
\frac{\d \Q^\omega}{\d \P} = \bigg(\frac{\d \Q^\omega}{\d {\widetilde\Q}^\omega_\psi}\bigg)\, \bigg( \frac{\d {\widetilde\Q}^\omega_\psi}{\d \P} \bigg),
$$
 we need to check that 
\begin{equation}\label{E4.5}
\E^{\Q}\bigg[\log \bigg(\frac{\d {\widetilde\Q}^\omega_\psi}{\d \P}\bigg)\bigg]=\frac{1}{2}\int_{R^3} |(\nabla\psi)(x)|^2 \d x
\end{equation}
Applying  Girsanov's formula, followed by It\^o's formula, we have 
\begin{align}\label{E4.6}
\log  \bigg(\frac{\d {\widetilde\Q}_\psi}{\d \P}\bigg) &=\int_0^1\bigg(\frac{\nabla\psi}{\psi}\bigg)(\omega(t)) \cdot \d\omega(t)
 -\frac{1}{2}\int_0^1 \bigg|\frac{\nabla\psi}{\psi}\bigg|^2(\omega(t))\d t\notag\\
&=\log\psi(\omega(1))-\log \psi(\omega(0))\\
&\qquad-\frac{1}{2}\int_0^1( \Delta\log \psi)(\omega(t))\d t-\frac{1}{2}\int_0^1 \bigg|\frac{\nabla\psi}{\psi}\bigg|^2(\omega(t))\d t
\end{align}
Then \eqref{E4.5} follows by taking expectation with respect to the stationary process with marginal density $\psi^2(x) \d x$ and integrating by parts. 
To derive the second assertion in the lemma, note that the quantity in \eqref{E4.6} reduces to
\begin{align*}
-\frac{1}{2}&\int_{\R^3}( \Delta\log\psi)(x)\psi^2(x) \d x-\frac{1}{2}\int_{\R^3}\frac {|\nabla\psi|^2(x)}{\psi^2(x)}\psi^2(x) \d x\\
&=\frac{1}{2}\int_{\R^3} |\nabla\psi|^2(x)\d x. 
\end{align*}
\end{proof}
\begin{remark}
In the context of proving Theorem \ref{thm1} for the Fr\"ohlich polaron, we will use Lemma \ref{lemma:entropy} when $\psi$ is the maximizer $\psi_0$ of the Pekar variational problem \eqref{Pekarfor}  and in this case 
$\log \psi_0$ has the required regularity  according to \cite{L76}.
Indeed, let us check the assumption on the maximizer $\psi_0$. Recall that
$$
g_0= \sup_{\heap{\psi\in H^1(\R^3)}{\|\psi\|_2=1}}\bigg\{\int\int_{\R^3\times\R^3} \frac{\psi^2(x)\psi^2(y)} {|x-y|} \, \d x \, \d y\,\, - \, \frac 12 \int_{\R^3} \d x |\nabla\psi(x)|^2 \bigg\}
$$
A simple perturbation argument shows that the maximizing function $\psi_0\in H^1(\R^3)$ satisfies the Euler-Lagrange equation 
\begin{equation}\label{EL}
\bigg(\Delta + 4 \int_{\R^3}\frac {\psi_0^2(y)}{|x-y|}\, \d y\bigg)\psi_0(x)= \lambda \psi_0(x).
\end{equation}
We multiply \eqref{EL} on both sides by $\psi_0(x)$, integrate over $\R^3$ and recall that $\int_{\R^3}\psi_0^2=1$, to see that
$$
\begin{aligned}
\lambda&=4 \int\int_{\R^3\times\R^3}\frac{\psi_0^2(x)\psi_0^2(y)}{|x-y|}-  \|\nabla\psi_0\|_2^2 \\
&=2 \bigg( 2\int\int_{\R^3\times\R^3}\frac{\psi_0^2(x)\psi_0^2(y)}{|x-y|}- \frac 12 \|\nabla\psi_0\|_2^2\bigg)
\geq 2 g_0 >0.
\end{aligned}
$$
It has been shown by Lieb \cite[Theorem 8, (iii)]{L76} that if $\lambda>0$ in \eqref{EL}, then the maximizing function $\psi_0\in C^\infty$, goes to zero at infinity and hence $\psi_0$ is a strong solution of \eqref{EL}. \qed
\end{remark}

\subsection{Coercivity of the variational formulas and tightness of maximizers.}

\begin{lemma}\label{lemma:epsPekar}
Consider the variational problem
\begin{equation}\label{epsPekar}
g(\eps)=\sup_{{\Q}\in{\cM}_{\mathrm{si}(\Omega_0)}}\bigg[ \E^{\Q}\bigg(\eps \int_0^\infty e^{-\eps t} \frac{1}{|\omega(t)-\omega(0)|} \d t\bigg)-H({\Q}|{\P})\bigg]
\end{equation}
Then the supremum is attained over a nonempty set ${\mathfrak m}_\eps$ of processes with stationary increments.
\end{lemma}

The proof of the above result is based on the following well-known fact which crucially exploits that the independence of the increments of the underlying measure $\P$.  

\begin{lemma}\label{lemma:uniform:entropy}
Fix any $C<\infty$. Then the set $\{\mathbb Q\in \Mcal_{\mathrm{si}}(\Omega_0)\colon H(\mathbb Q| \P) \leq C\}$ is uniformly tight in the weak topology.\qed
\end{lemma}

\begin{proof}[\bf{Proof of Lemma \ref{lemma:epsPekar}:}]
Let us consider any maximizing sequence $({\Q}_n)_n\subset \Mcal_{\mathrm{si}}(\Omega_0)$ such that
$$
\bigg[ E^{{\Q}_n}\big[\eps \int_0^\infty \e^{-\eps t} \frac{1}{|\omega(t)-\omega(0)|} dt\big]-H({\Q}_n|{\P})\bigg]\to g(\eps)
$$
By (\ref{E4.2}), for any  $\delta\in (0,1)$  there is a $C=C(\delta)$ such that for all $n$, 
$$
\E^{\mathbb Q_n}\bigg[\eps\int_0^\infty \e^{-\eps t}\frac{1}{|\omega(t)-\omega(0)|}\d t\bigg]\le \delta H({\Q_n}|{\P})+C(\delta)
$$
for all $\eps\le 1$. Combining the last two estimates imply that for some $C<\infty$, $\sup_n H({\Q}_n|{\P}) \leq C$, 
and then by Lemma \ref{lemma:uniform:entropy}, $\{\Q_n\}_n$ is also uniformly tight in $\Mcal_{\mathrm{si}}(\Omega_0)$. 
If
$$
{\Q}\stackrel{\mathrm{(def)}}= \lim_{n\to\infty} \Q_n
$$ 
is any subsequential limit point, the lower semicontinuity of $H(\cdot|\P)$ then 
implies that ${\Q}\in {\mathfrak m}_\eps$. 
\end{proof}

\begin{lemma}\label{lemma:meps:tight}
For any $\eps>0$ let $\mathfrak m_\eps$ be the maximizers of the variational formula \eqref{epsPekar}. Then for any $\eps_0>0$, 
$\cup_{0<\eps<\eps_0}{\mathfrak m}_\eps $ is uniformly tight in ${\mathcal M}_{\mathrm{si}}(\Omega_0)$.
\end{lemma}
\begin{proof}
Since  $g(\eps)\ge 0$ for all $\eps>0$, if ${\Q}\in \cup_{\eps>0}{\bf m}_\eps$, then again (\ref{E4.2}) for any  $\delta\in (0,1)$ implies 
$$
0\le E^{\Q} \bigg[\eps\int_0^\infty  \,\,  \frac{\e^{-\eps t}}{|\omega(t)-\omega(0)|} \d t\bigg]-H({\Q}|{\P})\le  ( \delta -1)H({\Q}|{\mathbb P})+ C(\delta)
$$ 
providing a uniform upper bound 
$$
H({\Q}|{\P})\le\frac{C(\delta)}{1-\delta}
$$
on $H({\Q}|{\ P})$, which together with Lemma \ref{lemma:uniform:entropy} proves the lemma. 
\end{proof}

\subsection{A key result: Identification of any limiting maximizer.}\label{sec-mainresult}

We will now prove a key result which identifies any limiting maximizer 
of the variational problem $\mathfrak m_\eps$ as the increments of the Pekar process.
Combined with Lemma \ref{lemma:meps:tight} it constitutes 
the main argument for the convergence of the Polaron measure $\widehat\P_\eps$ to the increments of the Pekar process. 
\begin{theorem}\label{thmain}
Let $\eps_n\to 0$ and ${\Q}_n\subset {\bf m}_{\eps_n}$ so that $\Q_n\weak\Q$ on $\Mcal_1(\Omega_0)$. Then $\Q$ is the distribution of the increments of a stationary process ${\widetilde\Q_\psi}\in \Mcal_1(\Omega_0\otimes \R^3)$ which is the diffusion corresponding to the generator 
$$
\frac{1}{2}\Delta+\frac{\nabla\psi}{\psi}\cdot\nabla,
$$
 with invariant density $\psi^2(x)dx$,  where $\psi$ is any maximizer of 
$$
\sup_{\|\psi\|_2=1}\bigg[ \int_{R^3}\int_{\R^3}\frac{\psi^2(x)\psi^2(y)}{|x-y|} \d x \d y-\frac{1}{2}\int_{\R^3} |(\nabla\psi)(x)|^2 \d x\bigg]
$$
\end{theorem}
\begin{remark}\label{rmk:thmain}
Recall that the maximizer $\psi$ for the above variational problem is unique only up to translations. This is reflected in the fact that ${\widetilde\Q_\psi} =\{\Q_{\psi_0}\star \delta_a\colon a\in \R^3\in \widetilde\Mcal_1(\Omega_0\otimes\R^3)$ is an orbit under translations and {\it{any representative}} $\Q_{\psi_0}\star\delta_a$ of that orbit has $\Q:=\lim_{n\to\infty}\Q_n\in\Mcal_1(\Omega_0)$ as the distribution of its increments. 
\end{remark}

\begin{proof}[Proof of Theorem \ref{thmain}:] The proof is carried out in several steps. 

{\bf Step 1.} It is known that (\cite{DV83}) 
\begin{equation}\label{E4.0}
\lim_{\eps\to 0} g(\eps)=g_0=\sup_{\|\psi\|_2=1}\bigg[ \int_{\R^3}\int_{\R^3}\frac{\psi^2(x)\psi^2(y)}{|x-y|} \d x \d y-\frac{1}{2}\int_{\R^3}
 |(\nabla\psi)(x)|^2 \d x\bigg]
\end{equation}
and therefore for  $({\Q}_n)\subset {\mathfrak m}_{\eps_n}$ we have
$$
\bigg[ \E^{{\Q}_n}\big[\eps_n \int_0^\infty e^{-\eps_n t} \frac{1}{|\omega(t)-\omega(0)|} \d t\big]-H({\Q}_n|{\P})\bigg]\to g_0.
$$

{\bf Step 2.} Note that, $x\mapsto 1/|x|$ is an even function in $\R^3$. Moreover, since each $\Q_n$ is a process with stationary increments (i.e. $\theta_t$ invariant measure on $\Omega_0$), 
for each $n$ and $\eps_n>0$, we have the identity, 
$$\int_0^\infty \d s \, \eps_n \e^{-\eps_n s} (\Q_n \theta^{-1}_s)(\cdot) =\int_0^\infty \d s \eps_n\e^{-\eps_n s} \Q_n(\cdot)= \Q_n(\cdot).$$ 
Using these two facts, we have 
\begin{equation}\label{eq1:mainth}
\begin{aligned}
\E^{\Q_n}\bigg[&\eps_n\int_0^\infty \frac{\e^{-\eps_n t} }{|\omega(t)-\omega(0)|}\d t\bigg]\\
&=\E^{\Q_n}\bigg[\frac{\eps_n}{2}\int_{-\infty}^\infty \frac{\e^{-\eps_n |t|} }{|\omega(t)-\omega(0)|}dt\bigg]\\
&=\E^{\Q_n}\bigg[\eps_n^2\int_0^\infty \int_0^\infty\frac{\e^{-\eps_n t-\eps_n s} }{|\omega(t)-\omega(s)|}\d t \d s\bigg]\\
&=\E^{\Pi_n}\bigg[  \int_{\R^3}\int_{\R^3}\frac{1}{|x_1-x_2|}{\widetilde\lambda}^{\ssup 2}(\d x_1){\widetilde\lambda}^{\ssup 2}(\d x_2) \bigg]
\end{aligned}
\end{equation}
where in the last display we used the definition for the averages $\nu_{\eps_n}(\cdot,\cdot)$ from \eqref{eq:nu} and as in \eqref{eq:Pi} we wrote $\Pi_n\in \Mcal_1(\Mcal_1(\Omega_0\otimes\R^3))$ for the distribution of $\nu_{\eps_n}$ under $\Q_n$. Then $\lambda^{\ssup 2}_n\in \Mcal_1(\R^3)$ is distributed according to the marginal of $\Pi_n$ on $\Mcal_1(\R^3)$ and moreover, the double integral 
$$
\int_{\R^3}\int_{\R^3} \frac 1 {|x_1-x_2|} \lambda^{\ssup 2}(\d x_1)\lambda^{\ssup 2}(\d x_2)
$$ is a function of the orbit $\widetilde\lambda^{\ssup 2}\in \widetilde\Mcal_1(\R^3)$, justifying the last identity in \eqref{eq1:mainth}.

Now according to Corollary \ref{cor:cocycle}, the projection  ${\widehat\Pi}_n$ of $\Pi_n$ on ${\cM_1}(\Omega_0)$
 is a uniformly tight family and therefore, by Theorem \ref{main2} and Corollary \ref{cor:main2}, $\Pi_n$ itself is uniformly tight on $\Mcal_1(\mathscr X)$.
Moreover, by Remark \ref{rmk:main2}, for any continuous function $V:\R^3 \to \R$ vanishing at infinity the function 
$$
\Psi(V,[\xi,\beta])=\sum_{{\widetilde \lambda}\in \xi}\int_{\R^3}\int_{\R^3} V(x_1-x_2){\widetilde\lambda}^{\ssup 2}(\d x_1){\widetilde\lambda}^{\ssup 2}(\d x_2) 
$$
is continuous on the metric space $(\mathscr X,\mathscr D)$. Furthermore,  the same argument as in the proof of Lemma \ref{lemma:epsPekar} provides a
uniform bound in $H(\Q_n|\P)$ allowing us to invoke the uniform estimate \eqref{E4.3} to control the unboundedness of the Coulomb potential. This estimate,
combined with the above continuity of $[\xi,\beta]\mapsto \Psi(V_\eta,[\xi,\beta])$ (with $V_\eta(x)=(|x|^2+\eta^2)^{-1/2}$) 
shows that if we take the limit 
$$
\Pi=\lim_n \Pi_n
$$
in the weak topology on $\Mcal_1(\mathscr X)$ along a subsequence, then \eqref{eq1:mainth} dictates 
\begin{equation}\label{E4.10}
\E^{\Pi}\bigg[   \sum_{{\widetilde \lambda}\in S}\int_{\R^3}\int_{\R^3} \frac{1}{|x_1-x_2|}{\widetilde\lambda}^{\ssup 2}(\d x_1){\widetilde\lambda}^{\ssup 2}(\d x_2)     \bigg]-H(\Q|\P)\ge g_0.
\end{equation}
In the above assertion we have additionally used the lower semi continuity of $H(\cdot|\P)$.

\medskip\noindent
{\bf Step 3.} Recall that typical elements of $\mathscr X$ are denoted by $[\xi,\beta]$, where 
$\beta \in \Mcal_1(\Omega_0)$ and $\xi=\{\widetilde\lambda_j\}_j$ and any $\lambda\in \xi$ is an orbit of  a measure with total mass at most $1$ on $\Omega_0\otimes \R^3$ and ${\widetilde\lambda}^{\ssup 1}$ and  ${\widetilde\lambda}^{\ssup 2}$ are their
projections on $\Omega_0$ and $\R^3$ respectively. Let $\Pi\in \Mcal_1(\mathscr X)$ denote the limit of $\Pi_n= \Q_n \nu_{\eps_n}^{-1}$ from Step 2 above. 
Then for any $[\xi, \beta]\in \mathscr X$ which is distributed according to $\Pi$, we have
$$
\sum_{{\widetilde\lambda}\in \xi}{\widetilde\lambda}^{\ssup 1}(\cdot)\le \beta(\cdot) \quad\Pi-\,\mbox{a.s.},
$$ 
and we denote the difference by $\beta_0(\cdot)=\beta(\cdot)- \sum_{{\widetilde\lambda}\in \xi}{\widetilde\lambda}^{\ssup 1}(\cdot)$ which is again 
$\theta_t$-invariant on $\Omega_0$. We can write              
$\beta$ as a convex combination of probability measures on $\Omega_0$:
$$
\begin{aligned}
&\beta(\cdot)=m({\beta_0})\,{\overline\beta}_0(\cdot)+\sum_{{\widetilde\lambda}\in \xi} m\big( {\widetilde\lambda^{\ssup 1}}\big){\overline \lambda}^{\ssup 1}(\cdot) \qquad\mbox{where}\\ &\overline\lambda^{\ssup 1}(\cdot)=\frac 1 {m( {\widetilde\lambda^{\ssup 1}})} \widetilde\lambda^{\ssup 1}(\cdot) \,\,\,\, \mbox{and}\quad 
\overline\beta_0(\cdot)= \frac 1 {m(\beta_0)} \beta_0(\cdot).
\end{aligned}$$ 
where $m( {\widetilde\lambda^{\ssup 1}})$ and $m(\beta_0)$  are the total masses of $ {\widetilde\lambda^{\ssup 1}}$ and $\beta_0$, respectively. 
This convex decomposition, combined with the linearity of the map $H(\cdot |\P)$, we have 
\begin{align}\label{E4.8}
H(\Q|\P)=\E^{\Pi}\bigg[ m(\beta_0) H({\overline\beta}_0|\P)+\sum_{{\widetilde\lambda}\in \xi}m({\widetilde\lambda})H({\overline\lambda}^{\ssup 1}|\P)\bigg]
\ge \E^{\Pi}\bigg[ \sum_{{\widetilde\lambda}\in \xi}m({\widetilde\lambda})H({\overline\lambda}^{\ssup 1}|\P)\bigg].
\end{align}
On the other hand since $m({\widetilde\lambda})\le 1$, we have 
\begin{align}\label{E4.9}
&\E^{\Pi}\bigg[   \sum_{{\widetilde \lambda}\in \xi}\int_{\R^3}\int_{\R^3} \frac{1}{|x_1-x_2|}{\widetilde\lambda}^{\ssup 2}(\d x_1){\widetilde\lambda}^{\ssup 2}(\d x_2)     \bigg] \\
&\le \E^{\Pi}\bigg[\sum_{{\widetilde \lambda}\in \xi} m({\widetilde\lambda})\int_{\R^3}\int_{\R^3} \frac{1}{|x_1-x_2|}{\overline\lambda}^{\ssup 2}(\d x_1){\overline\lambda}^{\ssup 2}(\d x_2)  \bigg].
\end{align}

\smallskip\noindent
{\bf Step 4.} Putting (\ref{E4.10}), (\ref{E4.8}) and (\ref{E4.9}) together  and observing that 
$ \sum_{{\widetilde \lambda}\in \xi}m({\widetilde\lambda})\le 1$, we have
\begin{align}
g_0&\le \E^{\Pi}\bigg[   \sum_{{\widetilde \lambda}\in \xi}\int_{\R^3}\int_{\R^3} \frac{1}{|x_1-x_2|}{\widetilde\lambda}^{\ssup 2}(\d x_1){\widetilde\lambda}^{\ssup 2}(\d x_2)     \bigg]-H(\Q|\P)\label{display1}\\
&\le \E^{\Pi}\bigg[\sum_{{\widetilde \lambda}\in \xi} [m({\widetilde\lambda})]^2\int_{\R^3}\int_{\R^3} \frac{1}{|x_1-x_2|}{\overline\lambda}^{\ssup 2}(\d x_1){\overline\lambda}^{\ssup 2}(\d x_2) -\sum_{{\widetilde\lambda}\in \xi}m({\widetilde\lambda})H({\overline\lambda}^{\ssup 1}|\P)\bigg]\\
&\le \E^{\Pi}\bigg[\sum_{{\widetilde \lambda}\in \xi} m({\widetilde\lambda})\int_{\R^3}\int_{\R^3} \frac{1}{|x_1-x_2|}{\overline\lambda}^{\ssup 2}(\d x_1){\overline\lambda}^{\ssup 2}(\d x_2) -\sum_{{\widetilde\lambda}\in \xi}m({\widetilde\lambda})H({\overline\lambda}^{\ssup 1}|\P)\bigg]\label{4.13}\\
&\le \sup_{{\widetilde\lambda}\in \widetilde{\cM}_{\mathrm s}(\Omega_0\otimes \R^3)}\bigg[\int_{\R^3}\int_{\R^3} \frac{1}{|x_1-x_2|}{\widetilde\lambda}^{\ssup 2}(\d x_1){\widetilde\lambda}^{\ssup 2}(\d x_2) -H({\widetilde\lambda}^{\ssup 1}|\P)\bigg]\\
&\le \sup_{\heap{\psi\in H^1(\R^3)}{\|\psi\|_2=1}}\bigg[\int_{\R^3}\int_{\R^3}\frac{1}{|x-y|}\psi^2(x)\psi^2(y)\d y-\frac{1}{2}\int_{\R^3} |\nabla\psi(x)|^2 \d x\bigg]\label{display2}\\
&=g_0\notag
\end{align}
In the penultimate  step we have  used the fact  from (\ref{E4.4}) that if $\widetilde\Q$ is a stationary proceess
with marginal distribution $\psi^2(x) \d x$, then 
\begin{equation}\label{eq:equal-entropy}
H({\widetilde\Q}^{\ssup 1}|\P)=\frac{1}{2}\int_{R^3} |(\nabla \psi)(x)|^2 dx+H({\widetilde\Q}^{\ssup 1}|{\Q}_\psi^{\ssup 1})\geq \frac{1}{2}\int_{\R^3} |(\nabla \psi)(x)|^2 dx
\end{equation}
and equality above holds only when $\widetilde\Q^{\ssup 1}=\Q_\psi$ which is the common distribution of increments of the stationary  diffusion with generator $\frac{1}{2}\Delta+\frac{\nabla \psi_y}{\psi_y}\cdot\nabla$.

\medskip\noindent
{\bf Step 5.}  Note that equality should hold at every step between \eqref{display1} and \eqref{display2}. Then it is easy to see that $\Pi$ is concentrated on a single orbit $[\{\widetilde\lambda\},\lambda]\in \mathscr X$ 
and $\lambda$ is the diffusion $\Q_\psi$ corresponding to a $\psi$ that maximizes (\ref{E4.0}). Indeed, recall Corollary \ref{cor:cocycle}. Since equality in (\ref{4.13}) forces $m({\widetilde\lambda})$ to be $0$ or $1$, but since  the sum over $\widetilde \lambda$ is at most $1$,
there is only one orbit $\widetilde \lambda$ in $\xi$, and by the equality in \eqref{display2} and previous remark regarding equality in \eqref{eq:equal-entropy}, 
the stationary process $\lambda\in \Mcal_{\mathrm s}(\Omega_0\otimes\R^d)$ must be the diffusion $\Q_\psi$ with generator 
$\frac{1}{2}\Delta+\frac{\nabla \psi_y}{\psi_y}\cdot\nabla$, with $\psi$ being a maximizer of (\ref{E4.0}). This argument concludes the proof of Theorem \ref{thmain}. 
\end{proof}

\section{Identification of the strong coupling Polaron and Kac-interactions: Proofs of Theorem \ref{thm1} and Theorem \ref{thm2}}\label{sec-result}

\subsection{Proof of Theorem \ref{thm1}.} 

Recall the definition of the Polaron measure $\widehat\P_{\eps,T}$ from \eqref{eq-Polaron-eps}. 
In \cite[Theorem 5.1]{MV18} it was shown that for all sufficiently small (or sufficiently large) $\eps>0$, the limit $\widehat\P_\eps=\lim_{T\to\infty}\widehat\P_{\eps,T}\in \Mcal_{\mathrm{si}}(\Omega_0)$ exists.
The following result provides an explicit identification of 
$\widehat\P_\eps$ in the strong coupling limit $\lim_{\eps\to 0} \widehat\P_\eps$ which, combined with \cite[Theorem 5.1]{MV18} will also complete the proof of Theorem \ref{thm1}.  
 \begin{theorem}\label{theorem}
$\lim_{\eps\to 0}\widehat\P_\eps=\Q_\psi$ exists and is equal to the distribution of increments of the stationary diffusion with generator 
$$
\frac{1}{2}\Delta+\frac{\nabla\psi_y}{\psi_y}\cdot\nabla
$$
\end{theorem}
As already remarked earlier, the distribution of increments of a stationary diffusion with the above generator is independent of $y\in \R^3$. 
Theorem \ref{thmain}, combined with Lemma \ref{lemma:meps:tight} will conclude 
the proof of the above result once we show that for $\eps>0$ (small enough, but fixed),  the infinite volume Polaron measure $\widehat\P_\eps \in \mathfrak m_\eps$.
\begin{prop}\label{prop}
Let $\widehat\P_\eps=\lim_{T\to\infty}\widehat\P_{\eps,T}$ and $\mathfrak m_\eps$ be the set of maximizers of the variational problem \eqref{epsPekar}. Then $\widehat\P_\eps\in \mathfrak m_\eps$.
\end{prop}
The proof of the above result depends on a {\it{strong large deviation principle}} for the distribution of the empirical process 
\begin{equation}\label{Rdef}
R_T(\omega, \cdot)=\frac 1 T\int_0^T \delta_{\theta_s\omega_T(\cdot)} \,\, \d s \in \Mcal_{\mathrm{si}}(\Omega_0)
\end{equation}
of increments, where 
$$
\omega_T(s)= 
\begin{cases}
\omega(s) \qquad\qquad\qquad\mbox{ if } 0\leq s \leq T,
\\
n \omega(T)+ \omega(r)\qquad \,\mbox{if} \,\,s=nT+r \quad n\in\N, 0\leq r<T.
\end{cases}
$$
Let $\mathbb Q_T= \P \, R_T^{-1} \in \Mcal_1(\Mcal_{\mathrm{si}}(\Omega_0))$ be the distribution of the empirical process under three-dimensional 
Brownian increments $\P$. Then 
 \begin{lemma}\label{lemma:LDP}
The family $(\mathbb Q_{T})_{T>0}$  satisfies a strong large deviation principle as $T\to\infty$
in the space of probability measures on $\Mcal_{\mathrm{si}}(\Omega_0)$ with rate function $H(\cdot|\P)$. In other words, 
\begin{equation}\label{eq-ldp-statement}
\begin{aligned}
\liminf_{T\to\infty}\frac 1 T \log \mathbb Q_T(G) \geq -\inf_{\mathbb Q\in G} H(\mathbb Q|\P) \qquad \forall\,\,\,G\subset \Mcal_{\mathrm{si}}\,\,\,\mbox{open}\\
\limsup_{T\to\infty}\frac 1 T \log \mathbb Q_T(F) \leq -\inf_{\mathbb Q\in F} H(\mathbb Q|\P) \qquad \forall\,\,\,F\subset \Mcal_{\mathrm{si}}\,\,\,\mbox{closed}
\end{aligned}
\end{equation}
\end{lemma}
\begin{proof}
The proof of the lower bound for all open sets $G\subset \Mcal_{\mathrm{si}}$ and the upper bound for all compact sets $K\subset \Mcal_{\mathrm{si}}$ in \eqref{eq-ldp-statement} 
follows directly from the arguments of \cite{DV4} modulo slight changes, and the details are omitted.
To strengthen the upper bound to all closed sets $C\subset \Mcal_{\mathrm{si}}$, 
it suffices to show exponential tightness for the distributions $\mathbb Q_T$ that requires that for any $\ell>0$, existence of a compact set $K_\ell\subset\Mcal_{\mathrm{si}}$
so that, 
\begin{equation}\label{eq-1-pf-thm1}
\limsup_{T\to\infty} \frac 1 T\log \mathbb Q_T[K_\ell^{\mathrm c}] \leq -\ell.
\end{equation}
To prove the above claim, for $i=1,2,3$, if we set
$$
\|\omega_i^\star\|= \sup_{0\leq s \leq t\leq 1} \frac{|\omega_i(s)-\omega_i(t)|}{|s-t|^{1/4}}, 
$$
then by Borell's inequality, for  some constants $C_1, C_2>0$, 
$$
\P\big[\|\omega_i^\star\| > \lambda\big] \leq C_1 \exp\bigg[-\frac {\lambda^2}{2C_2}\bigg],
$$
and consequently, $\E^{\P}[\e^{\|\omega_i^\star\|}]<\infty$. Then 
$$
\limsup_{T\to\infty} \frac 1 T \log \E^{\mathbb Q_T}\big[\e^{T \sum_{i=1}^3 \|\omega_i^\star\|}\big] <\infty,
$$ 
and the desired exponential tightness \eqref{eq-1-pf-thm1} follows readily, proving the requisite upper bound in \eqref{eq-ldp-statement} for all closed sets. 
\end{proof}

The following lemma will complete the proof of Proposition \ref{prop}. 
\begin{lemma}\label{lemma:tight}
Fix any $\eps>0$. Then the distributions 
$\widehat\P_{\eps,T} \,R_T^{-1}$
 of the empirical process of increments under the Polaron measure satisfies a strong large deviation principle in $\Mcal_1(\Mcal_{\mathrm{si}}(\Omega_0))$ with rate function 
$$
\Q\mapsto g(\eps)- \bigg[\E^\Q\bigg(\eps \int_0^\infty \frac{\eps \e^{-\eps r}\,\d r}{|\omega(r)-\omega(0)|}\bigg)- H(\mathbb Q|\P)\bigg].
$$
Moreover any limit point of $\widehat\P_{\eps,T} \,R_T^{-1}$ (as $T\to\infty$) lies in $\mathfrak m_\eps$.
\end{lemma}
\begin{proof}
It follows from Lemma \ref{lemma:LDP} that $\lim_{T\to\infty}\frac 1 T\log Z_{\eps,T}=g(\eps)$ (we can again invoke Lemma \ref{lemma-Coulomb1} to control the unboundedness of the Coulomb potential in $\R^3$) 
and the required strong large deviation principle for $\widehat\P_{\eps,T} \,R_T^{-1}$ follows immediately. 
In particular, the upper bound of this large deviation implies that for any open neighborhood $U(\mathfrak m_\eps)$ of $\mathfrak m_\eps$, 
$\limsup_{T\to\infty} \frac 1 T\log\widehat\P_{\eps,T}[R_T\notin U(\mathfrak m_\eps)]<0$, which together with compactness of the level sets of $H(\cdot|\P)$ 
forces all limit points of $\widehat\P_{\eps,T} \,R_T^{-1}$ to be supported in $\mathfrak m_\eps$.
\end{proof}
\smallskip
\noindent{\bf{Proof of Proposition \ref{prop} and Theorem \ref{theorem}:}} Since the actual limit $\widehat\P_\eps=\lim_{T\to\infty} \widehat\P_{\eps,T}$ exists (\cite[Theorem 5.1]{MV18}) it follows, together with Lemma \ref{lemma:tight} that 
$\widehat\P_\eps\in \mathfrak m_\eps$. Theorem \ref{thmain} then concludes the proof of Theorem \ref{theorem}.\qed

\subsection{Proof of Theorem \ref{thm2}.} 

Recall the definition of the path measure $\widehat\P^{\ssup V}_{\eps,T}$ from \eqref{PV} defined w.r.t. any $V:\R^d\to \R$ which is rotationally symmetric, continuous and vanishes at infinity. 
For such a function $V$ in Theorem \ref{thm2} we also assume that the variational problem 
$$
\sup_{\heap{\psi\in H^1(\R^d)}{\|\psi\|_2=1}} \bigg[\int\int_{\R^{2d}} V(x-y) \psi^2(x) \psi^2(y) \d x \d y - \frac 12 \big\|\nabla \psi\big\|_2^2\bigg]
$$
admits a smooth maximizer $\psi^{\ssup V}$ which is unique, modulo translations in $\R^d$. For the proof of Theorem \ref{thm2}, we can essentially repeat the argument for proving Theorem \ref{thm1}, except that 
we can no longer appeal to \cite[Theorem 5.1]{MV18} to justify the existence of the thermodynamic limit 
\begin{equation}\label{thermo}
\lim_{T\to\infty} \widehat\P^{\ssup V}_{\eps,T}=: \widehat\P^{\ssup V}_\eps.
\end{equation}
The reason for this is that the method in our earlier article \cite{MV18} hinges upon the Gaussian representation 
$\frac 1 {|x|}= \sqrt{\frac 2\pi} \int_0^\infty \e^{-\frac{u^2|x|^2}2} \d u$ 
of the particular choice of the Coulomb potential
$V(x)=\frac 1 {|x|}$ in $d=3$. This representation is no longer available for a general interaction $V$ being considered in Theorem \ref{thm2}. However, in \cite{M17} we considered Gibbs measures with translation invariance of the form 
$$
\d\widehat\P^{\ssup{H}}_{\alpha,T}:= \frac 1 {Z_T} \exp\bigg\{ \alpha \int_{-T}^T\int_{-T}^T H(t-s, \omega(t)-\omega(s)) \d s \d t\bigg\} \d \P
$$
defined w.r.t. the law $\P$ of increments of $d$-dimensional Brownian paths. In the main result of \cite{M17} it has been shown that for any coupling parameter $\alpha>0$, the thermodynamic limit 
$\lim_{T\to\infty} \widehat\P^{\ssup{H}}_{\alpha,T}$ exists and a central limit theorem also holds for $\lim_{T\to\infty} \widehat\P^{\ssup{H}}_{\alpha,T}\big[\frac{\omega(T)-\omega(-T)}{\sqrt{2T}}\in \cdot\big]$, provided that 
the interaction $H$ satisfies 
\begin{equation}\label{H}
\sup_{x\in \R^d} | H(t,x) | \leq C(1+|t|)^{-\gamma} \qquad\mbox{for some} \,\, \gamma>2, \qquad\mbox{and } C\in (0,\infty)
\end{equation}
In the context of Theorem \ref{thm2}, we have 
$$
H(t,x)=\eps \e^{-\eps |t| }V(x) \qquad\mbox{with } \|V\|_\infty \leq C\in (0,\infty).
$$
Clearly, the above choice satisfies the desired requirement \eqref{H} for any fixed $\eps>0$ and 
we can appeal to \cite{M17} to justify \eqref{thermo}. Given the existence of the limit $\widehat\P^{\ssup V}_\eps$ in 
\eqref{thermo} and repeating the large deviation arguments leading to Proposition \ref{prop} (with $V(\cdot)$ replacing the Coulomb potential $1/|\cdot|$), we have 
that $\widehat\P^{\ssup V}_\eps$ is a maximizer of the variational problem \eqref{epsPekar}, again with $V(\omega(t)-\omega(0))$ replacing the Coulomb potential $\frac 1 {|\omega(t)-\omega(0)|}$ therein. Now we can exactly repeat the arguments for the proof of Theorem \ref{thm1}
to conclude the proof of Theorem \ref{thm2} as well. Certainly, owing to the boundedness of $V$ some arguments now simplify (e.g. the truncation arguments needed  to handle the singularity of the Coulomb potential carried out in Lemma \ref{lemma-Coulomb1} will no longer be required for a bounded function $V$). 

{\bf{Acknowledgement.}} The authors are very grateful to Herbert Spohn for helpful discussions and comments on a first draft of the manuscript.

 \end{document}